\documentclass{article}
\usepackage{amsmath}
\usepackage{amssymb}
\usepackage{amsthm}
\usepackage{amsfonts}
\usepackage{diagbox}
\usepackage[all]{xy}
\usepackage{pgf,tikz}
\usetikzlibrary{arrows}

\newtheorem{theorem}{Theorem}[section]
\newtheorem{proposition}[theorem]{Proposition}
\newtheorem{lemma}[theorem]{Lemma}
\newtheorem{Theorem}[theorem]{Theorem}
\newtheorem{Lemma}[theorem]{Lemma}
\newtheorem{corollary}[theorem]{Corollary}
\newtheorem{Corollary}[theorem]{Corollary}
\newtheorem{Proposition}[theorem]{Proposition}
\newtheorem{definition}[theorem]{Definition}
\newtheorem{Definition}[theorem]{Definition}

\newtheorem*{remark}{Remark}
\newtheorem*{example}{Example}

\DeclareFontFamily{U}{rsf}{}
\DeclareFontShape{U}{rsf}{m}{n}{
  <5> <6> rsfs5 <7> <8> <9> rsfs7 <10-> rsfs10}{}
\DeclareMathAlphabet{\mathscr}{U}{rsf}{m}{n}

\DeclareMathAlphabet{\mathgth}{U}{euf}{m}{n}

\DeclareFontFamily{U}{cyr}{}
\DeclareFontShape{U}{cyr}{m}{n}{
  <5> wncyr5 <6> wncyr6 <7> wncyr7 <8> wncyr8 <9> wncyr9 <10-> wncyr10}{}
\DeclareMathAlphabet{\mathcyr}{U}{cyr}{m}{n}

\makeatletter
\def\operator@font{\sf}
\makeatother

\DeclareMathOperator{\Spec}{Spec}

\newcommand{\Hom}{{\mathsf{Hom}}}

\newcommand{\ra}{\rightarrow}

\newcommand{\C}{\mathbb{C}}

\newcommand{\Z}{\mathbf{Z}}

\newcommand{\field}[1]{\mathbb{#1}}

\newcommand{\F}{\field{F}}

\renewcommand{\phi}{\varphi}

\title{Algebraic geometry over the residue field of the infinite place}
\author{
M\'arton Hablicsek\\[2mm]
Mathematics Department, University of Pennsylvania\\
\texttt{mhabli@math.upenn.edu}
\and
M\'at\'e L. Juh\'asz\\[2mm]
Alfr\'ed R\'enyi Institute of Mathematics\\
Hungarian Academy of Sciences\\
\texttt{juhasz.mate.lehel@renyi.mta.hu}
}

\DeclareMathOperator{\Sym}{Sym}
\DeclareMathOperator{\infplace}{\infty}
\newcommand{\Finfext}[1]{\mathbb{F}_{\infty^{#1}}}

\newcommand{\plus}{\mathop{\dot+}}

\newcommand{\infelem}{\omega}

\newcommand{\Finfty}{\mathbb{F}_{\infplace}}

\renewcommand{\Z}{\mathbb{Z}}

\begin{document}

\definecolor{xdxdff}{rgb}{0.49019607843137253,0.49019607843137253,1.}
\definecolor{qqqqff}{rgb}{0.,0.,1.}

\maketitle

\begin{abstract}
Nikolai Durov introduced the theory of generalized rings and schemes to study Arakelov geometry in an alternative algebraic framework, and introduced the residue field at the infinite place, $\mathbb{F}_{\infty}$. We show an elementary algebraic approach to modules and algebras over this object, define prime congruences, show that the polynomial ring of $n$ variables is of Krull dimension $n$, and derive a prime decomposition theorem for these primes.
\end{abstract}

\section{Introduction}

In the category of schemes, the initial object is $\Spec\mathbb{Z}$, which is not a complete variety. Suren Yurevich Arakelov introduced the concept of Arakelov geometry in \cite{arakelov1974} and \cite{arakelov1975}, by introducing Hermitian metrics on holomorphic vector bundles over the complex points of an arithmetic surface. Arakelov geometry can be used to study diophantine equations from a geometric point of view. For instance, it can be used to prove certain results over number fields which are known over function fields (see \cite{Fal} and \cite{Voj} for examples). In Nikolai Durov's doctoral dissertation (\cite{durov2007}), Durov introduces a new approach to Arakelov geometry, the theory of generalized rings and fields, and uses them, among others, to construct a completion of $\Spec\mathbb{Z}$.

To understand the completion, consider that the divisor of a rational function on a complete curve is always of degree zero. Put differently, the sum of all valuations of a rational function at all points of the curve gives zero. The analoguous formulation for $\Spec\mathbb{Z}$ states that the product of all valuations on $\mathbb{Q}$ of a rational number is always one. Recall that the valuations on $\mathbb{Q}$ are of two kinds: an Archimedean valuation $|.|$ and for all primes $p$ a non-Archimedean valuation $|.|_p$. Just like the valuations of a complete curve, each non-Archimedean valuation on $\mathbb{Q}$ corresponds to a closed point of $\Spec\mathbb{Z}$, however, the Archimedean valuation, called for analogical reasons the \emph{valuation at the infinite place or infinity}, is missing from $\Spec\mathbb{Z}$.

Durov uses this idea to complete $\Spec\mathbb{Z}$, and, among others, he defines the residue field corresponding to the valuation at infinity. In general, for a prime $p$, we introduce $\mathbb{Z}_{(p)}=\{q\in\mathbb{Q}\mid |q|_p\le 1\}$ and the open unit ball $U_p:=\{q\in\mathbb{Z}_{(p)}\mid |q|_p<1\}$, and define the residue field as the quotient $\mathbb{F}_p=\mathbb{Z}_{(p)}/U_p$. For the Archimedean valuation, $\mathbb{Z}_{(\infplace)}=[-1,1]\cap\mathbb{Q}$ and $U_{\infplace}=(-1,1)\cap\mathbb{Q}$, and $\Finfty$ is, intuitively, the closed interval $[-1,1]$ with its interior identified as a single element, $0$. This is in fact the underlying set of the object that Nikolai Durov refers to as $\mathbb{F}_{\infty}$. In this paper we investigate algebras over this generalized ring. Explicitly, this generalized ring has three elements: $-1$, $0$, $1$, equipped with three operations: $\plus$, $\cdot$ and a unary $-$ so that
\begin{itemize}
\item $\plus$ is idempotent: $a\plus a=a$ for every element $a$,
\item $0$ is an absorbing element: $a\plus 0=0$ for every element $a$,
\item $\plus$ is commutative and associative, 
\item $-(1)=-1$, $-(-1)=1$, $-0=0$,
\item $-1+1=0$,
\item $\cdot$ is commutative, associative, and distributive with respect to $\plus$,
\item $1$ is a multiplicative identity, 
\item $0$ is absorbing with respect to $\cdot $ as well,
\item $(-1)\cdot (-1)=1$.
\end{itemize}
For instance, we clearly see that $\Finfty$ is not a ring in the usual sense, it lacks of an additive identity, and, as a consequence, of additive inverses. On the other hand, it is not very far from being a ring itself. One purpose of this paper is to show that algebraic geometry over $\Finfty$ looks similar to algebraic geometry over finite fields.

The paper consists of two parts. First, instead of using Nikolai Durov's full machinery, we give a gentle introduction to the theory of algebras and modules over $\Finfty$. In Section 2, we introduce $\Finfty$-fields, including finite extensions, modules and algebras, and motivate using polynomial rings as the ring of functions. By looking at modules as semilattices, with the addition functioning as a meet-operation, in Section 3 we show how these can be extended into lattices (\ref{thm:finislat}, \ref{thm:infislat}), and use this to understand dual modules (\ref{thm:findual}, \ref{thm:infdual}) and $\Hom$-modules (\ref{thm:homstruct}), at first for finite modules, then using topology, to infinite modules. In fact, any module can be embedded into one where infinite sums and joins exist. In Section 4, we examine the theory of congruences and kernels. In particular, it turns out that for any ideal there is always a maximal congruence whose kernel is the ideal, which can be identified by a separability condition (\ref{thm:maxcongmodule} and \ref{thm:maxcongalgebra}). Then we turn to congruences in $\Finfty$-fields, where the congruence is characterized completely by the equivalence class of $1$, neatly mirroring classical ring theory with the equivalence class of $0$.

Second, we take the first steps towards algebraic geometry over $\Finfty$. Our theory is mainly motivated by a novel approach by D\'aniel Jo\'o and Kalina Mincheva (\cite{joo2014}, \cite{joo2015}). One of the key ideas in these papers is that prime congruences are more natural objects to study than prime ideals. The authors study tropical geometry using prime congruences instead of prime ideals, and they prove a version of Hilbert's Nullstellensatz in the realm of tropical geometry. We follow this key idea and we define prime congruences in algebras over $\Finfty$ and, among others, we show that the polynomial ring of $n$ variables has Krull dimension $n$ (see Corollary \ref{cor:krull}), and we derive a prime decomposition theorem (see Theorem \ref{thm:primedec}). As a consequence, we bring in line the theory of modules and algebras over $\Finfty$ with the theory of classical finite fields.

\bigskip

\textbf{Acknowledgement:} We thank Kalina Mincheva and D\'aniel Jo\'o for patiently explaining their work to us. Most of the techniques we use in the second part of the paper appear in some form in their work.

We would also like to thank an anonymous referee for their invaluable comments.

\section{Modules and algebras over $\Finfty$}
\label{sec:mod}

First, let us recall what the object $\Finfty$ is (see \cite{durov2007}, \textbf{5.1.16}).

\begin{definition}
The \textbf{generalized field} or \textbf{field} $\Finfty$ has as set the elements $\{-1,0,1\}$, with three operations: a unary $-$, and a binary $\plus$ and $\cdot$, with the following properties:
\begin{itemize}
\item $(a\plus b)\plus c=a\plus (b\plus c)$, $a\plus b=b\plus a$;
\item $0\plus a=a\plus 0=0$, hence $0$ is an \textbf{absorbing element};
\item $a\plus a=a$;
\item $a\plus(-a)=0$;
\item $-(1)=-1$, $-(-1)=1$, $-0=0$;
\item $(a\cdot b)\cdot c=a\cdot (b\cdot c)$, $a\cdot b=b\cdot a$;
\item $1\cdot a=a\cdot 1=a$;
\item $(-1)\cdot(-1)=1$.
\end{itemize}
\end{definition}

Note in particular that the operator $\plus$ is unlike the typical addition in that \textbf{$0$ is an absorbing element} and that \textbf{$-a$ is not the additive inverse of $a$}. The notations were chosen this way because of the way Durov defined these structures in \cite{durov2007}, and since in his development, $\plus$ is derived as a convex combination. Durov chose the notation $*$ for this operator, but we wanted to preserve its connection to addition. Also note that $0$ is absorbing for both operations.

\bigskip

Intuitively, a module over the field at infinity, $\Finfty$, corresponds to the faces of a symmetric polyhedron, where the binary operation is \emph{the smallest face containing both}. Then the field $\Finfty$ itself is the digon $[-1,1]$ with three elements: $1$, $-1$ and $0$, and the binary operation $\plus$ is \emph{the element between}. Although there are modules that cannot be realized as actual symmetric polyhedra, this can be a useful visualization.

In the following definitions, we will denote by indices the arity of the operations.
In \cite{durov2007}, modules over generalized rings are described in \textbf{4.3.7}. Here we give a more direct description for $\Finfty$.

\begin{definition}
\label{def:module}
An \textbf{$\Finfty$-module} is a structure $(V,0_0,-_1,\plus_2)$ such that:
\begin{itemize}
\item $(a\plus b)\plus c=a\plus (b\plus c)$, $a\plus b=b\plus a$;
\item $a\plus a=a$, $a\plus (-a)=0$;
\item $-(a\plus b)=(-a)\plus (-b)$; $-(-a)=a$.
\end{itemize}
A \textbf{submodule}, \textbf{congruence}, \textbf{quotient module} and \textbf{module homomorphism} are defined as usual.
\end{definition}

Note that contrary to usual conventions, $0$ is not a neutral element, and $-a$ is not an additive inverse.

\begin{proposition}
$0$ is an absorbing element in an $\Finfty$-module.
\end{proposition}

\begin{proof}
$0=a\plus (-a)=a\plus a\plus (-a)=a\plus 0$.
\end{proof}

\begin{example}
The set $\Finfty=\{-1,0,1\}$ is an $\Finfty$-module, defined uniquely by the module axioms.
\end{example}

\begin{proposition}
\label{thm:natorder}
An $\Finfty$-module has a natural partial order defined by $a\le b$ if $a\plus b=a$, with $0$ being the smallest element, i.e. $0\plus x=0$.
\end{proposition}

\begin{proof}
Reflexivity arises from idempotence, symmetry from commutativity. If $a\le b\le c$, then $a\plus b=a$ and $b\plus c=b$, hence $a\plus c=(a\plus b)\plus c=a\plus b=a$, therefore $a\le c$.
\end{proof}

\begin{definition}
A \textbf{maximal element} $x$ is such that there is no such $y$ that $x<y$, i.e. $x\plus y=x$ if and only if $x=y$.
A \textbf{minimal element} $x$ is such that $y<x$ only for $y=0$, i.e. $x\plus y=y$ if and only if $x=y$ or $y=0$.
\end{definition}

We will also need rings over $\Finfty$:

\begin{definition}
\label{def:algebra}
An \textbf{$\Finfty$-algebra} or \textbf{ring} is an $\Finfty$-module $A$ with a semigroup structure $(A,\cdot_2)$ such that
\begin{itemize}
\item $a\cdot0=0\cdot a=0$;
\item $-a\cdot b=a\cdot-b=-(a\cdot b)$;
\item $a\cdot(b\plus c)=(a\cdot b)\plus (a\cdot c)$, $(a\plus b)\cdot c=(a\cdot c)\plus (b\cdot c)$.
\end{itemize}
A \textbf{subalgebra}, \textbf{congruence}, \textbf{quotient algebra} and \textbf{algebra homomorphism} are defined as usual.

A \textbf{unital algebra} is a monoid structure $(V,1_0,\cdot_2)$ such that $a\cdot 1=1\cdot a=a$. An \textbf{invertible element} $a$ is such that there is an $a^{-1}$ such that $a\cdot a^{-1}=1$, and the group of invertible elements is denoted by $A^\times$. The unital algebra is a \textbf{division algebra}  if $(V\setminus\{0\},1,\cdot)$ is a group, and a \textbf{$\Finfty$-field} if it is a commutative group.
\end{definition}

Recall once again that $0$ is not a neutral element for addition:

\begin{proposition}
$0$ is an absorbing element in an $\Finfty$-algebra.
\end{proposition}

\begin{proposition}
In a finite unital algebra, all invertible elements are incomparable.
\end{proposition}

\begin{proof}
First, $1$ is not less than any invertible element, since if $1<a$ for $a$ invertible, then $a^i$ gives an infinite increasing sequence. Similarly $1$ is not greater than any invertible element. Then, if $a$ and $b$ are invertible, and $a<b$, we may multiply both sides by $a^{-1}$, which preserves the inequality by the distributivity of multiplication. Hence $1<a^{-1}b$, which is a contradiction.
\end{proof}

\begin{corollary}
\label{cor:findivalg}
In a finite division algebra, all non-zero elements are minimal and maximal. Hence for any $a$, $b\in A$, $a\plus b=0$ unless $a=b$.
\end{corollary}

\begin{remark}
Note that our usage of \textbf{$\Finfty$-algebra} is more in line with what Durov refers to as \emph{unary algebras} over $\Finfty$, as defined in \textbf{5.1.15} in \cite{durov2007}, referred to as such because what we describe as elements, he refers to them as \emph{unary operations} (\textbf{4.3.9}). A proof for this will be sketched in \ref{prop:Durov}. What he refers to as algebras, described in \textbf{5.1.9}, are more general structures, and may have additional operations of higher arity.
\end{remark}

\begin{example}
~
\begin{itemize}
\item The field $\Finfty$ is a commutative division $\Finfty$-algebra.

\item $\Finfext{k}$ is the $\Finfty$-field generated by $\zeta_k$ such that $\zeta_k^k=-1$ (for the definition of a freely generated algebra, see \ref{def:freealg} and \ref{prop:freealg}). Its elements are $\{\zeta_k^i\mid i\in\{0,1,\dots,2k-1\}\}$, and by \ref{cor:findivalg} we have $a\plus b=0$ unless $a=b$. These fields can be embedded as subsets of $\mathbb{C}$, and the diagram below shows the $k=3$ case.

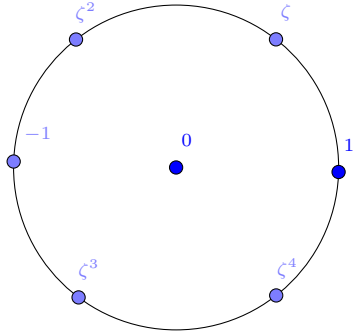
\hspace{1cm}
\begin{tikzpicture}[line cap=round,line join=round,>=triangle 45,x=1.0cm,y=1.0cm]
\clip(2.62,3.06) rectangle (10.22,8.52);
\draw(7.,6.) circle (2.160833172644293cm);
\begin{scriptsize}
\draw [fill=qqqqff] (7.,6.) circle (2.5pt);
\draw[color=qqqqff] (7.14,6.36) node {$0$};
\draw [fill=qqqqff] (9.16,5.94) circle (2.5pt);
\draw[color=qqqqff] (9.3,6.3) node {$1$};
\draw[color=black] (5.96,7.6) node {};
\draw [fill=xdxdff] (5.6693162168405475,7.702492487277535) circle (2.5pt);
\draw[color=xdxdff] (5.8,8.06) node {$\zeta^2$};
\draw [fill=xdxdff] (5.703500096413424,4.271333461884566) circle (2.5pt);
\draw[color=xdxdff] (5.84,4.64) node {$\zeta^3$};
\draw [fill=xdxdff] (4.840620343485974,6.079243290147304) circle (2.5pt);
\draw[color=xdxdff] (5.16,6.44) node {$-1$};
\draw [fill=xdxdff] (8.327995025105867,7.704590629240366) circle (2.5pt);
\draw[color=xdxdff] (8.46,8.06) node {$\zeta$};
\draw [fill=xdxdff] (8.330683783159452,4.297507512722465) circle (2.5pt);
\draw[color=xdxdff] (8.48,4.66) node {$\zeta^4$};
\end{scriptsize}
\end{tikzpicture}

\item Given $k|\ell$, there is a natural embedding of $\Finfext{k}$ into $\Finfext{\ell}$, given by $\zeta_k=\zeta_\ell^{\ell/k}$. The field $\Finfext{\infty}:=\varinjlim \Finfext{k}$ can be embedded as a subset of $\C$. Its elements are the roots of unity and 0.

\item The underlying sets of all these finite fields can be embedded into $\mathbb{C}$ as the $k$th roots of unity. The Euclidean closure of $\Finfext{\infty}$ in $\mathbb{C}$ is given as $\overline{\Finfext{\infty}}:=\mathbb{T}\cup \{0\}=\{z\in\mathbb{C}\mid |z|=1\}\cup \{0\}$. Using the multiplication on $\mathbb{C}$ and defining addition as $a\plus b=0$ unless $a=b$, this gives a field structure to the set $\overline{\Finfext{\infty}}$.

\item We may also consider the extension with $\zeta^k=1$. In this case, its elements are $\pm\zeta^i$ for $i\in\{0,\dots,k\}$. These and $\Finfext{k}$ will appear later as quotients of the polynomial ring $\Finfty[x]$ in \ref{thm:polyprimes}.

\item More generally, given a group $G$ and a field $\mathbb{F}$, the \textbf{group algebra} $\mathbb{F}[G]$ consists of elements $0$ and $\lambda g$ with $\lambda\in\mathbb{F}^\times$ and $g\in G$, with the law of addition that $\lambda g\plus \lambda'g'=(\lambda\plus \lambda')g$ if $g=g'$, otherwise $0$, and $\lambda g\cdot\lambda'g'=(\lambda\lambda')(gg')$. The previous example is in fact $\Finfty[\mathbb{Z}/k\mathbb{Z}]$.

\item There is a more general way to construct fields. Consider a commutative group $G$ with an injective map $f\colon\Finfty^\times\to G$. Then the set $G\cup\{0\}$ has a natural $\Finfty$-field structure, defined via the group operation as multiplication, $a\plus b=0$ unless $a=b$, and $-a=f(-1)a$. This generalizes group algebras, with $G=\mathbb{F}^\times\times H$ for $\mathbb{F}[H]$. 

\item For an example of a $\Finfty$-field where the order is non-trivial, consider the set $\{\pm x^i\mid i\in\mathbb{Z}\}\cup\{0\}$ with the addition $x^i\plus x^j=x^i$ if $i\ge j$, and $x^i\cdot x^j=x^{i+j}$.

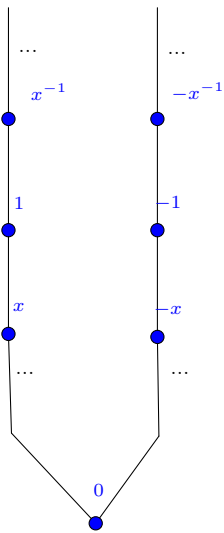
\definecolor{qqqqff}{rgb}{0.,0.,1.}
\begin{tikzpicture}[line cap=round,line join=round,>=triangle 45,x=1.0cm,y=1.0cm]
\clip(1.24,1.88) rectangle (12.84,10.34);
\draw (7.02,2.48)-- (5.9,3.68);
\draw (7.02,2.48)-- (7.86,3.64);
\draw (5.9,3.68)-- (5.86,5.);
\draw (7.86,3.64)-- (7.84,4.96);
\draw (5.86,5.)-- (5.86,6.38);
\draw (7.84,4.96)-- (7.84,6.38);
\draw (5.86,6.38)-- (5.86,7.86);
\draw (7.84,6.38)-- (7.84,7.86);
\draw (5.86,7.86)-- (5.86,9.34);
\draw (7.84,7.86)-- (7.84,9.34);
\begin{scriptsize}
\draw [fill=qqqqff] (7.02,2.48) circle (2.5pt);
\draw[color=qqqqff] (7.06,2.92) node {$0$};
\draw[color=black] (6.74,3.46) node {};
\draw[color=black] (7.74,3.04) node {};
\draw [fill=qqqqff] (5.86,5.) circle (2.5pt);
\draw[color=qqqqff] (6.,5.36) node {$x$};
\draw[color=black] (6.08,4.48) node {$...$};
\draw [fill=qqqqff] (7.84,4.96) circle (2.5pt);
\draw[color=qqqqff] (7.98,5.32) node {$-x$};
\draw[color=black] (8.14,4.48) node {$...$};
\draw [fill=qqqqff] (5.86,6.38) circle (2.5pt);
\draw[color=qqqqff] (6.,6.74) node {$1$};
\draw[color=black] (6.22,5.86) node {};
\draw [fill=qqqqff] (7.84,6.38) circle (2.5pt);
\draw[color=qqqqff] (7.98,6.74) node {$-1$};
\draw[color=black] (8.2,5.84) node {};
\draw [fill=qqqqff] (5.86,7.86) circle (2.5pt);
\draw[color=qqqqff] (6.4,8.22) node {$x^{-1}$};
\draw[color=black] (6.22,7.28) node {};
\draw [fill=qqqqff] (7.84,7.86) circle (2.5pt);
\draw[color=qqqqff] (8.38,8.22) node {$-x^{-1}$};
\draw[color=black] (8.2,7.28) node {};
\draw[color=black] (6.12,8.76) node {$...$};
\draw[color=black] (8.1,8.74) node {$...$};
\end{scriptsize}
\end{tikzpicture}
\end{itemize}
\end{example}

Modules and algebras can also be considered over other fields.

\begin{definition}
Assume $\mathbb{A}$ is an $\Finfty$-algebra. An $\mathbb{A}$-module $M$ is an $\Finfty$-module with a binary operation $\mathbb{A}\times M\to M$, such that
\begin{itemize}
\item $a\cdot (b\cdot m)=(a\cdot b)\cdot m$ for $a$, $b\in\mathbb{A}$, $m\in M$;
\item $a\cdot (m\plus n)=a\cdot m\plus a\cdot n$ and $(a\plus b)\cdot m=a\cdot m\plus b\cdot m$ for $a$, $b\in\mathbb{A}$, $m$, $n\in M$;
\item $(-a)\cdot m=-(a\cdot m)=a\cdot(-m)$ for $a\in\mathbb{A}$, $m\in M$;
\item $0\cdot m=a\cdot 0=0$ for $a\in\mathbb{A}$, $m\in M$;
\item If $\mathbb{A}$ is unital, we further postulate $1\cdot m=m$ for $m\in M$.
\end{itemize}
An $\mathbb{A}$-algebra $M$ is an $\mathbb{A}$-module that is also an $\Finfty$-algebra, such that
\begin{itemize}
\item $a\cdot (m\cdot n)=(a\cdot m)\cdot n=m\cdot(a\cdot n)$ for $a\in\mathbb{A}$, $m$, $n\in M$.
\end{itemize}
\end{definition}

\begin{lemma}
The category theoretic free module $\mathbb{A}(n)$ over $\mathbb{A}$, generated by $x_1,\dots,x_n$ consists of elements $\bigoplus_{i\in I} \lambda_ix_i$ for $I\subseteq\{1,\dots,n\}$, $\lambda_i\in\mathbb{A}\setminus\{0\}$, and the element $0$. The operation $\bigoplus_{i\in I}\lambda_ix_i\plus\bigoplus_{j\in J}\mu_jx_j=\bigoplus_{i\in I\cap J}(\lambda_i\plus\mu_i)x_i\oplus\bigoplus_{i\in I\setminus J}\lambda_ix_i\oplus\bigoplus_{j\in J\setminus I}\mu_jx_j$ with $0x_i=0$ and $0\plus a=0$ for any $a$.
\end{lemma}

This is a consequence of \ref{prop:freemod} that we will announce later.

\begin{theorem}
\label{prop:Durov}
Any $\Finfty$-algebra $\mathbb{A}$ (defined as in \ref{def:algebra}) can be given a generalized ring structure, as defined in \cite{durov2007}, and the categories of modules and algebras over it are equivalent to the categories of modules and unary algebras over $\mathbb{A}$ respectively, the later as defined in \cite{durov2007}. Also, $\mathbb{A}$ is an $\Finfty$-algebra, in the sense of \cite{durov2007}.
\end{theorem}

\begin{proof}
Here we will sketch the proof, referring to sections in \cite{durov2007}, denoted by bold numerals.

Durov defines generalized rings in \textbf{5.1} as a commutative algebraic monad. In \textbf{4.3}, an algebraic monad $\Sigma$ is described by defining $\Sigma(n)$, which is a free module generated by $n$ elements, as shown in \textbf{4.6.5}, and several natural maps $\mu_n^{(k)}\colon\Sigma(k)\times\Sigma(n)^k\to\Sigma(k)$, which satisfy certain natural relations. In \textbf{4.3.7}, modules $M$ are defined through maps $\alpha^{(k)}\colon\Sigma(k)\times M^k\to M$ with a similar description.

In our case, we already have $\mathbb{A}(n)$, consisting of terms in variables $x_1$, {\dots}, $x_n$. We may define $\mu_n{(k)}(t,s_1,\dots,s_n)$ and $\alpha^{(k)}(t,\sigma_1,\dots,\sigma_n)$ for $t\in\mathbb{A}(n)$, $s_i\in\mathbb{A}(k)$ and $\sigma_i\in M$ by replacing each $x_i$ by $s_i$ and $\sigma_i$, respectively, in $t$. Then the relations and commutativity may be checked directly. It can be checked that the homomorphisms are identical to those defined in \ref{def:module} as well by noting that all elements of $\mathbb{A}(n)$ can be written using $+$ and the variables $x_1$, {\dots}, $x_n$.

It is also worth noting that in general, every universal algebraic variety containing no relations (i.e. a category of algebras satisfying a certain set of equalities) gives rise to an algebraic monad, hence the variety of $\mathbb{A}$-modules is also trivially an algebraic monad, giving equivalent categories. The only condition needing to be checked separately is the commutativity, which is a very simple corollary of \textbf{5.1.7} and the fact that $\mathbb{A}$ is generated by operations of arity at most $2$.

Since any $\Finfty$-algebra $\mathbb{A}$, according to our definition, can be generated by adding elements $a_i\in\mathbb{A}$ and relations $f_i=g_i$ between terms in $a_i$, these are exactly what Durov refers to as \emph{unary algebras} in \textbf{5.1.15}, as all of $a_i$, $f_i$ and $g_i$ are unary operations, according to \textbf{4.3.9}. Furthermore, as announced in \textbf{5.3.8} and \textbf{5.3.13}, the category of unary $\mathbb{A}$-algebras is equivalent to the category of algebras in the category of $\mathbb{A}$-modules, and by \textbf{5.3.9}, this is identical to how we defined $\mathbb{A}$-algebras and homomorphisms in \ref{def:algebra}.

In \textbf{5.1.9}, an algebra $\Sigma$ over $\Lambda$ is defined by the existence of a ring homomorphism $\Lambda\to\Sigma$, and given an $\mathbb{A}$-algebra $M$, we have a natural map $\mathbb{A}\to M$.
\end{proof}

Henceforth we will consider only unital algebras.

\smallskip

Let $\mathbb{F}$ denote an $\Finfty$-field. Recall that $\mathbb{F}^\times=\mathbb{F}\setminus\{0\}$.

\begin{definition}
The \textbf{dimension} of an $\mathbb{F}$-module $M$, denoted by $\dim M$, is the maximal length of a chain of decreasing elements.
\end{definition}

\begin{example}
Fix an integer $n\ge 2$, and let us denote the vertices of a regular $2n$-gon $P$ by $v_i$ for $i\in\{0,\dots,2n-1\}$ and the edges connecting $v_i$ to $v_{i+1}$ by $e_i$. The lattice of all faces $\{P,v_i,e_i\mid i\in\{0,\dots,2n-1\}\}$ has a natural $\Finfty$-module structure with $P=0$, $-v_i=v_{i\pm n}$ and $v_i\plus v_{i+1}=e_i$, otherwise $x\plus y=0$. It has dimension $2$, since the sum of pairwise uncomparable elements is non-zero if and only if there is a single term, or two adjacent vertices.
\end{example}

\begin{example}
Given a convex, symmetric polyhedron $P$ in $\mathbb{R}^n$ such that it has a non-empty interior, the set of faces has a natural $\Finfty$-module structure with $P=0$ and $f\plus f'$ is the smallest face that contains $f$ and $f'$. This has dimension $n$, as the longest chain of faces contains one of each dimension, including $P$.
\end{example}

\begin{definition}
For two modules $M_1$ and $M_2$, $M_1+M_2$ is the \textbf{coproduct}, whose elements are of the form $m_1\in M_1$, $m_2\in M_2$ and $m_1\oplus m_2$, with $0_{M_1}=0_{M_2}=0_{M_1}+0_{M_2}$ identified. The coproduct of several modules is denoted by $\sum_{i=1}^n M_i$.\hfill\break
$M_1\times M_2$ is the \textbf{Cartesian product}, with operations evaluated coordinate-wise. The product of several modules is denoted by $\prod_{i=1}^n M_i$. The \textbf{free module} generated by $n$ elements is given by $\mathbb{F}+\cdots+\mathbb{F}$. The free module generated by elements of a set $S$ is denoted by ${\cal F}(S)$.\hfill\break
The set of module-homomorphisms is denoted by $\Hom(M_1,M_2)$.
\end{definition}

\begin{proposition}\label{prop:freemod}
The coproduct, Cartesian product and free module are the category theoretical coproduct, product and free object. Also, the free module ${\cal F}(S)$ generated by a set $S$ is isomorphic to the coproduct $\sum_{s\in S}\mathbb{F}$ (Note that if $S$ is infinite, the coproduct is infinite as well).
\end{proposition}

\begin{proof}
These are all trivial consequences of theorems in universal algebra.
\end{proof}

\begin{proposition}
$\Hom(M_1,M_2)$ has a natural $\mathbb{F}$-structure.
\end{proposition}

\begin{proof}
The pointwise sum and scalar multiple of two homomorphisms is a homomorphism.
\end{proof}

\begin{remark}
This notion of $\Hom$ is identical to the definition in \cite{durov2007}, \textbf{4.6.3}, further elaborated in \textbf{5.3.1}.
\end{remark}

\begin{definition}
Given two $\mathbb{F}$-modules $M_1$ and $M_2$, the \textbf{tensor product} $M_1\otimes M_2$ is a module with a natural bilinear map $M_1\times M_2\to M_1\otimes M_2$ where $M_1\times M_2$ is the set of pairs, such that for any bilinear map $M_1\times M_2\to N$ for some other module $N$, there is exactly one map $M_1\otimes M_2\to N$ that makes the diagram $M_1\times M_2\to M_1\otimes M_2\to N$ commute.
\end{definition}

\begin{proposition}
Given two $\mathbb{F}$-modules $M_1$ and $M_2$, $M_1\otimes M_2$ exists and is unique. It is generated by elements of the form $m_1\otimes m_2$ with $m_1\in M_1$ and $m_2\in M_2$. Furthermore $\cdot\otimes M$ is a covariant functor.
\end{proposition}

\begin{proof}
This is a classical theorem from universal algebra, and the proof is identical to the case of vector spaces. Uniqueness can be checked via diagram chasing. We can prove the existence by constructing $M_1\otimes M_2$ explicitly. Let us consider the free module generated by pairs $m_1\otimes m_2$ with $m_1\in M_1$ and $m_2\in M_2$, and quotienting by the congruence generated by $(m_1\plus m_1')\otimes m_2\sim m_1\otimes m_2\plus m_1'\otimes m_2$, $\lambda(m_1\otimes m_2)\sim(\lambda m_1)\otimes m_2\sim m_1\otimes(\lambda m_2)$. Naturally, any bilinear map $M_1\times M_2\to N$ extends uniquely into a map $M_1\otimes M_2\to N$.

Finally, for the functoriality, given a map $f\colon A\to B$, we need to construct $A\otimes M\to B\otimes M$. We may define the bilinear map $A\times M\to B\otimes M$ defined by $(a,m)\to f(a)\otimes m$, and this extends into the desired map. Identity and composition can be checked as usual.
\end{proof}

\begin{proposition}\label{prop:tensorsum}
$\Hom(M_1,\Hom(M_2,M_3))\cong\Hom(M_1\otimes M_2,M_3)$, $A\otimes\mathbb{F}\cong A$, $(A+B)\otimes C\cong (A\otimes C)+(B\otimes C)$.
\end{proposition}

\begin{proof}
Since elements of $\Hom(M_1,\Hom(M_2,M_3))$ correspond naturally to bilinear maps $M_1\times M_2\to M_3$, there is a natural embedding to $\Hom(M_1\otimes M_2,M_3)$. On the other hand, a map $\varphi\colon M_1\otimes M_2\to M_3$ restricts to a map $\varphi'\colon M_1\times M_2\to M_3$ that is bilinear. Therefore $\varphi'(m)$ for $m\in M_1$ is a homomorphism, and $\varphi'$ is a homomorphism from $M_1$.

The second one is trivial, since an element $a\otimes\lambda\in A\otimes\mathbb{F}$ is equal to $(\lambda a)\otimes 1$.

For the third one, there is a natural bilinear map from $(A+B)\times C$ to $A\otimes C+B\otimes C$, defined as $(a\oplus b,c)\to(a\otimes c)\oplus(b\otimes c)$, which extends to a unique map $(A+B)\otimes C\to A\otimes C+B\otimes C$. On the other hand, since $A\otimes C+B\otimes C$ is the coproduct, and there are maps from $A\otimes C$ and $B\otimes C$ to $(A+B)\otimes C$, these define a unique map $A\otimes C+B\otimes C\to(A+B)\otimes C$. It can be shown that the composition in either direction is the identity using the universality of the tensor product and the coproduct.
\end{proof}

\begin{remark}
Tensor products are in fact defined in \cite{durov2007}, \textbf{5.3.5} using adjunction.
\end{remark}

\begin{definition}
Let $M$ be an $\mathbb{F}$-module. Then $\bigotimes^n M$ or $M^{\otimes n}$ denotes the tensor product $M\otimes\dots\otimes M$, and $\bigotimes^0 M=\mathbb{F}$. The congruence generated by
$\mathbf{a}_1\otimes a_2\otimes a_3\otimes \mathbf{a}_4\sim\mathbf{a}_1\otimes a_3\otimes a_2\otimes \mathbf{a}_4$
for $\mathbf{a}_1\in M^{\otimes n_1}$, $\mathbf{a}_4\in M^{\otimes n_2}$ for $n_1+n_2+2=n$ and $a_2$, $a_3\in M$ is the kernel of the surjective map $M^{\otimes n}\to\Sym^n M$ that defines the \textbf{symmetric power} of $M$. Furthermore, there are natural maps $\bigotimes^{n_1}M\otimes\bigotimes^{n_2}M\to\bigotimes^{n_1+n_2}M$ and $\Sym^{n_1}M\otimes\Sym^{n_2}M\to\Sym^{n_1+n_2}M$.

Given a module $M$, the \textbf{tensor algebra} and the \textbf{symmetric ring} of $M$ are \textbf{graded rings} defined as $\bigotimes M:=\sum_{n=0}^\infty \bigotimes^n M$ and $\Sym M:=\sum_{n=0}^\infty \Sym^n M$, respectively, and multiplication is given by the above maps.
\end{definition}

\begin{remark}
In \cite{durov2007}, the tensor algebra and symmetric ring are defined identically, in \textbf{5.3.18} and \textbf{5.3.19}, respectively.
\end{remark}

\smallskip

To do algebraic geometry, we need to construct the coordinate ring of an affine space. First let us introduce the polynomial ring in $n$ variables, following the definition in \cite{durov2007}, \textbf{5.3.22}.

\begin{definition}
\label{def:freealg}
The \textbf{polynomial ring} in $n$ variables $x_1$, {\dots}, $x_n$ over $\mathbb{F}$, denoted by $\mathbb{F}[x_1,\dots,x_n]$, is
defined as the symmetric product $\Sym{\cal F}(\{x_1,\dots,x_n\})$.
\end{definition}

\begin{lemma}
The polynomial ring generated by $\{x_1,\dots,x_n\}$ is isomorphic to
$$\sum_{\mu\colon\{1,\dots,n\}\to \mathbb{N}} (x_1^{\mu_1}\cdot\dots\cdot x_n^{\mu_n})\mathbb{F},$$
as a module, and the ring structure may be defined as the extension of multiplication of monomials $x_1^{\mu_1}\cdot\dots\cdot x_n^{\mu_n}$.
\end{lemma}

\begin{proof}
First let us write $M={\cal F}(\{x_1,\dots,x_n\})$ as the coproduct $\sum_{i=1}^n x_i\mathbb{F}$.
By \ref{prop:tensorsum}, we may write $\bigotimes^k M$ as $\sum_{\sigma\in S} (x_{\sigma(1)}\otimes\dots\otimes x_{\sigma(k)})\mathbb{F}$ where $S$ is the set of maps from $\{1,\dots,k\}$ to $\{1,\dots,n\}$.
By symmetrizing, we get that $\Sym M$ is isomorphic to the above description.
\end{proof}

\begin{proposition}
\label{prop:freealg}
The polynomial ring in $n$ variables is the \textbf{free object} generated by $\{x_1,\dots,x_n\}$ in the category of $\mathbb{F}$-algebras. Then every finitely generated algebra is the quotient of a polynomial ring.
\end{proposition}

\begin{proof}
Proving this universal property for the polynomial ring amounts to showing that given an algebra $A$ with certain elements $a_i\in A$ for $1\le i\le n$, there is a unique map $\mathbb{F}[x_1,\dots,x_n]\to A$ that sends $x_i$ to $a_i$.

Every monomial of the form $x_1^{\mu_1}\cdot\dots\cdot x_n^{\mu_n}$ has a well-defined image. Since the polynomial ring is the category theoretical coproduct of the modules $(x_1^{\mu_1}\cdot\dots\cdot x_n^{\mu_n})\mathbb{F}$, each freely generated by a monomial, maps from these modules define a (unique) map from the polynomial ring. Since this map preserves multiplication on monomials, it preserves all products.

For the second part of the theorem, we may send the variables $x_i$ to the generators of the algebra.
\end{proof}

Then we may define the affine space as the set of homomorphisms from the polynomial ring to the underlying field, analoguously with classical fields. Note that this is not what we will refer to as the algebraic variety, since it is not a \emph{spectrum} of a ring.

\begin{definition}
The \textbf{affine space} of dimension $n$, $\mathcal{A}^n$, is the set consisting of ring homomorphisms from $\mathbb{F}[x_1,\dots,x_n]$ to $\mathbb{F}$. We will call $\mathbb{F}[x_1,\dots,x_n]$ the \textbf{ring of functions} of $\mathcal{A}^n$.
\end{definition}

\begin{lemma}
The points of the affine space $\mathcal{A}^n$ are in a bijection with the elements of the module $A:=\prod_{i=1}^n\mathbb{F}=\mathbb{F}\times\dots\times\mathbb{F}$, and the homogeneous degree $1$ elements of the polynomial ring $\mathbb{F}[x_1,\dots,x_n]$ are in a canonical bijection with $A^*:=\Hom(A,\mathbb{F})$.
\end{lemma}

\begin{proof}
A map $\mathbb{F}[x_1,\dots,x_n]\to\mathbb{F}$ is defined by giving the images of the variables $x_1$, {\dots}, $x_n$, since the polynomial ring is the free object. If the images are denoted by $\varepsilon_1$, {\dots}, $\varepsilon_n$, this gives an element $(\varepsilon_1,\dots,\varepsilon_n)\in A$, and since the $\varepsilon_i$ may be arbirtary, this defines the bijection.

The degree $1$ part of $\mathbb{F}[x_1,\dots,x_n]$ is isomorphic to the free coproduct of modules $\sum_{i=1}^n x_i\mathbb{F}$. To see that its elements give linear maps on the module $A$, it is sufficient to check for the generators of this module, $f=x_i$. In fact, it is enough to check for a single variable $f=x_1$. Consider two points $p$, $q\in\mathcal{A}^n$, identified by $\check p:=(p(x_1),\dots,p(x_n))$ and $\check q:=(q(x_1),\dots,q(x_n))\in A$. Their images under $x_1$ are $p(x_1)$ and $q(x_1)$. Given a linear combination $\lambda\check p\plus\mu\check q\in A$, its image is $\lambda p(x_1)\plus\mu q(x_1)$ under $x_1$, since $\plus$ is evaluated coordinate-wise in the product $A$. This is equal to the linear combination of the images.
Later we will prove \ref{prop:proddual}, and as a consequence the finite module $A^*$ is isomorphic to the free coproduct $\sum_{i=1}^n \mathbb{F}$, and it is in fact generated by the linear functions $x_i$.
\end{proof}

As is usual in the theory of classical fields, any vector space can be given a natural affine space structure by forgetting the origin of the space.
Unfortunately, contrary to the theory of classical fields, the $n$-dimensional module is not unique.
However, since any finitely generated ring is a quotient of a polynomial ring, the different spaces all arise as affine subvarieties of $\mathcal{A}^n$, and we may study them as such.
Nevertheless, we will present a general definition for the coordinate ring of such spaces arising from a module, and we will call them \emph{affine cones}.

With the symmetric product, we may define the ring of functions over an affine cone arising from a module $M$. We want the homogeneous linear functions on this affine cone to be isomorphic to $M^*$, as was the case for $M=\prod_{i=1}^n\mathbb{F}$, which motivates our following definition.

\begin{definition}
The \textbf{ring of functions} over $M$, as an affine cone, is defined as $\mathbb{F}[M]:=\Sym M^*$. Then the associated \textbf{affine cone}, denoted by $|M|$, is the set consisting of ring homomorphisms $\mathbb{F}[M]\to\mathbb{F}$.
\end{definition}

\begin{example}
For instance, consider the module $M$ generated by $u$, $v$, such that $u+v=u-v=0$. This is a module consisting of $5$ elements: $\{0,u,v,-u,-v\}$. Then the ring of functions is isomorphic to $\mathbb{F}[x,y]/\{x+y\sim 0,x-y\sim 0\}$. It is a simple verification that $|M|$ is in a bijection with $M$, hence it gives rise to an affine cone that is not isomorphic to $\mathcal{A}^n$.

\vspace{-1cm}
\hspace{0.2cm}
\definecolor{qqqqff}{rgb}{0.,0.,1.}
\definecolor{cqcqcq}{rgb}{0.7529411764705882,0.7529411764705882,0.7529411764705882}
\begin{tikzpicture}[line cap=round,line join=round,>=triangle 45,x=1.0cm,y=1.0cm]

\clip(0.8619601995019549,0.15443313252972632) rectangle (12.240214382225686,9.355124876835639);
\begin{scriptsize}
\draw [fill=qqqqff] (6.,5.) circle (2.5pt);
\draw[color=qqqqff] (6.13911601873072,5.362927073069641) node {$0$};
\draw [fill=qqqqff] (4.,5.) circle (2.5pt);
\draw[color=qqqqff] (4.1376110360467013,5.362927073069641) node {$-u$};
\draw [fill=qqqqff] (8.,5.) circle (2.5pt);
\draw[color=qqqqff] (8.140621001414738,5.362927073069641) node {$u$};
\draw [fill=qqqqff] (6.,7.) circle (2.5pt);
\draw[color=qqqqff] (6.13911601873072,7.36443205575366) node {$v$};
\draw [fill=qqqqff] (6.,3.) circle (2.5pt);
\draw[color=qqqqff] (6.13911601873072,3.3614220903856222) node {$-v$};
\end{scriptsize}
\end{tikzpicture}
\end{example}

\vspace{-2cm}

Note that contrary to the case of classical rings, an affine cone based on a module \textbf{preserves its origin}: any non-constant function of $\Sym M^*$ evaluates to $0$ in the origin, while as we will see in later chapters, there is at least one linear function for every other point that does not send it to $0$.

\section{Ordered structure}

Recall from Proposition \ref{thm:natorder} that an $\mathbb{F}$-module for a given $\Finfty$-field $\mathbb{F}$ has a natural partial order. It has many of the usual properties of an ordered algebraic structure:

\begin{proposition}
Given an $\mathbb{F}$-module $M$ and elements $a$, $b$, $c\in M$ such that $a\le b$, we have $a\plus c\le b\plus c$ and $ac\le bc$, and if $c\le d\in M$, we have $a\plus c\le b\plus d$ and $ac\le bd$. In particular, if $a\le b$ then $-a\le -b$ as well. Also, if $a\le b$ and $a\le c$ then $a\le b\plus c$.
\end{proposition}

\begin{proof}
These are elementary consequences of the definition, the idempotence of additivity and distributivity.
\end{proof}

Such a module is in fact a \textbf{semilattice} with respect to the meet-operation, $\plus_2$. It has clearly no lattice structure, since if $a$, $-a\le x$ for some $x$ and $a\ne0$, we would have $a\le -x$ and $a\le x\plus -x=0$, a contradiction. This can be salvaged by the introduction of a largest element.

\begin{definition}
For an $\mathbb{F}$-module $M$, we will denote by $\overline M$ the set $\{\infelem\}\cup M$ where $\infelem>a$ for all $a\in M$, and call it the \textbf{order closure} or \textbf{closure} of $M$. We may extend the operations as partial operations via $a\plus \infelem=a$ and $\lambda\infelem=\infelem$ for $\lambda\ne0$. $0\infelem$ is undefined.
\end{definition}

\begin{definition}
An $\mathbb{F}$-module $M$ is called a \textbf{lattice} if the order closure $\overline M$ of $M$ has a lattice structure. This means that for any $a$, $b\in M$, there is an upper bound $a\cup b\in\overline M$.
\end{definition}

\begin{theorem}
\label{thm:finislat}
A finitely generated $\mathbb{F}$-module $M$ is a lattice.
\end{theorem}

\begin{proof}
If $M$ is finitely generated, then there is a finite set $G$ of generators, and the sum $a\cup b:=\sum\limits_{\substack{x\in G\\ x\ge a,b}} x$ is finite and well-defined. Now assume that $d\ge a$, $b$. Since $M$ is finitely generated, there is a finite subset $G(d)\subseteq G$ such that $d=\sum_{x\in G(d)}x$. Then $x\ge a$, $b$ for all $x\in G(d)$, hence $d\ge a\cup b$. Therefore $a\cup b$ is a lower bound to $d$.
\end{proof}

The lattice structure permits us to construct a \textbf{dual lattice}, one where $\plus $ and $\cup$ switch places. For finite modules, it turns out that the dual lattice has a natural algebraic meaning: the module of linear functions to $\mathbb{F}$. This can be expressed through the natural duality.

\begin{definition}
For an $\mathbb{F}$-module $M$, there is a \textbf{natural duality} $(\cdot,\cdot)\colon (M\setminus\{0\})\times M\to\mathbb{F}$ defined as $(a,b)=\varepsilon$ if $b\ge \varepsilon a$, and $0$ if no such $\varepsilon$ exists.
We extend it to $(a,b)\colon\overline M\times\overline M\setminus\{(0,0),(\infelem,\infelem)\}\to\overline{\mathbb{F}}$, defined via $(\infelem,a)=0$, $(a,\infelem)=\infelem$, $(a,0)=\infelem$.
\end{definition}

\begin{proposition}
The natural duality on $M$ is well defined, and in particular, the map $a^*:=(a,\cdot)$ is a homomorphism from $M$ to $\mathbb{F}$ for all $a\in \overline M\setminus\{0\}$. Also, $(\mu a,b)=\mu^{-1}(a,b)$ for $\mu\in\mathbb{F}^\times$.
\end{proposition}

\begin{proposition}
\label{thm:findual}
If $M$ is a finite module, then there is a natural order-reversing bijection between $\overline M$ and $\overline{M^*}$ where $M^*:=\Hom(M,\mathbb{F})$, and the addition on $M^*$ is given by $\cup$.
\end{proposition}

\begin{proof}
Any element $a\in \overline M\setminus\{0\}$ gives a natural map $a^*\colon M\to\mathbb{F}$. Consider a function $f\colon M\to\mathbb{F}$ that is not trivially zero. Then the sum $F=\sum\{c\mid f(c)=1\}$ is well-defined, and $F^*=f$.
\end{proof}

\smallskip

Now let us look at how to define dual modules for infinite modules. The following propositions show the na\"\i ve way of looking at homomorphic maps to the base field.

\begin{proposition}
For a given $\mathbb{F}$-module $M$, homomorphic maps $f\colon M\to \mathbb{F}$ are in a bijection with filters $F$, given by $a\in F\Leftrightarrow f(a)=1$.
\end{proposition}

\begin{proposition}
The set of filters $F$ on a given $\mathbb{F}$-module $M$ form an $\mathbb{F}$-module $M^F$, given by $0_{M^F}=\emptyset$, $F_1\plus F_2=F_1\cap F_2$ and $\varepsilon\cdot F=\{\varepsilon^{-1} a\mid a\in F\}$ for $\varepsilon\in \mathbb{F}^{\times}$.
\end{proposition}

\begin{proposition}
There is a natural injection $M\to(M^F)^F$, given by sending $a$ to $\hat a:=\{F\in M^F\mid a\in F\}$. This is not always a bijection.
\end{proposition}

This is in contrast to the finite case, when $M$ and $(M^*)^*$ are isomorphic. By introducing a weak form of topology, we can define a better concept of dual module.

\begin{definition}
A \textbf{principal filter} of an $\mathbb{F}$-module $M$ is a filter of the form $F_a=\{x\in M\mid a\le x\}$, and it is said to be \textbf{generated by $a$}.
We say that an $\mathbb{F}$-module $M$ has a \textbf{topology with respect to filters} if all principal filters are closed.
A \textbf{topology with respect to the order} is such that $F_a$ and all sets of the form $L_a=\{x\in M\mid a\ge x\}$ are closed.
\end{definition}

When no topology is specified, we may assume the discrete topology where all filters are closed.

Given a filter $F$ and an element $a\in M$, we will denote by $F(a)=\varepsilon\in\mathbb{F}^\times$ if $a\in\varepsilon F$, and $F(a)=0$ if no such $\varepsilon$ exists. This is compatible with the natural duality for principal filters: $F_a(b)=(a,b)$.

\begin{definition}
If $M$ is an $\mathbb{F}$-module $M$ with topology with respect to filters, let us denote by $M^*$ the set of closed filters. It is called the \textbf{dual module} of $M$. $M^*$ is an $\mathbb{F}$-module, where $0_{M^*}:=\emptyset_M$, $F_1\plus F_2=F_1\cap F_2$, $\lambda\cdot F=\{\lambda^{-1}\cdot a\mid a\in F\}$. Its \textbf{filter-topology} has a closed basis given by $C_a:=\{\Phi\in M^*\mid a\in\Phi\}$ for all $a\in M$, and is a topology with respect to filters. Its \textbf{weak topology} is generated by all $C_a$ and their complements $\overline{C_a}$, and is a topology with respect to the order.
\end{definition}

\begin{proposition}
Given a descending sequence $(x_i)_{i\in I}$ in $M^*$ indexed by a directed set $I$, the intersection is an accumulation point with respect to either the filter-topology or the weak topology.
\end{proposition}

\begin{proof}
Consider the point $x=\bigcap_{i\in I}x_i$. To prove that it is an accumulation point, we need to show that every open set containing $x$ contains all $x_i$ for $i\ge i_0$ for some $i_0\in I$. Since an open basis is given by the complement $\overline{C_a}$ of $C_a$, all open sets containing $x$ must contain an open set of the form $\bigcap_{a\in S}\overline{C_a}$ for a finite set $S\subseteq M$, hence it is enough to show this statement for open sets of this form. Since $x\in C:=\bigcap\overline{C_a}$, this is equivalent to $a\not\in x$ for all $a\in S$, and since $x$ is the intersection of all $x_i$, there is an $i_a$ for all $a$ such that $a\not\in x_{i_a}$. Being a directed set, $I$ contains an index $i_0$ such that $i_0\ge i_a$ for all $a\in S$, because $S$ is finite, and since $(x_i)$ is descending, $a\not\in x_i$ for all $i\ge i_0$ and $a\in S$. Hence $x_i\in C$. The case of weak topology is similar, but open sets of the form $C_a$ may also appear in the intersection $C$.
\end{proof}

\begin{corollary}
Given a closed filter $\Phi$ in either the filter-topology or the weak topology on $M^*$, and a subset $S\subseteq\Phi\subseteq M^*$, the infinite intersection $\bigcap_{s\in S}s$ exists and is an element of $\Phi$.
\end{corollary}

\begin{proof}
Since the elements of $M^*$ are closed filters on $M$, their intersection is also a closed filter, hence an element of $M^*$. We only need to show that it is an element of $\Phi$. This can be proven by transfinite recursion on the cardinality of $S$. When $S$ is finite, it is a trivial consequence of the definition of a filter. When $S$ is infinite, let $I$ be the powerset of $S$, and for $i\subseteq S$, let $s_i:=\bigcap_{s\in i}s$. By transfinite recursion, all $s_i$ exist and are contained in $\Phi$. Then $(s_i)_{i\in I}$ is a descending sequence, and the intersection $\bigcap_{s\in S}s$ is its accumulation point. Therefore it is in $\Phi$.
\end{proof}

\begin{theorem}
Given an $\mathbb{F}$-module, there is a natural embedding $M\to(M^*)^*$, given by $\hat a:=\{F\in M^*\mid a\in F\}$.
\end{theorem}

\begin{proof}
This is a homomorphism since $\widehat{a\plus b}=\hat a\cap\hat b$, an embedding since $\Phi_a:=\{x\in M\mid a\le x\}$ is such that $\Phi_a\in\hat b$ if and only if $a\le b$, so if $\hat a$ and $\hat b$ both contain $\Phi_a$ and $\Phi_b$, then $a=b$.
\end{proof}

\begin{definition}
A module $M$ is \textbf{complete} if every closed filter is principal.
\end{definition}

\begin{theorem}
\label{thm:infislat}
\label{thm:infdual}
A complete module is a lattice, and it admits an order reversing isomorphism to its dual.
\end{theorem}

\begin{proof}
The second part of the statement is a trivial consequence of the definition of a complete module.
Consider a complete module $M$ and two elements $a$, $b\in M$. The filter $F:=F_a\cap F_b$ is closed and contains all elements $x$ such that $a$, $b\le x$. Since $M$ is complete, $F$ is principal, generated by $c$, hence $a$, $b\le c$ and for all $x\in F$, $c\le x$.
\end{proof}

\begin{proposition}
Every finite module is complete with respect to the discrete topology.
\end{proposition}

\begin{proof}
Every filter is finite, hence it is generated by the sum of its elements.
\end{proof}

\begin{theorem}
The module $M^*$ is complete and is a lattice, with respect to either the filter-topology or the weak topology.
\end{theorem}

\begin{proof}
Given a closed filter $\Phi\subseteq M^*$, the intersection $\bigcap\Phi$ is an element of $\Phi$. Since $\bigcap\Phi\le a$ for all $a\in\Phi$, $\Phi$ is the principal filter generated by $\bigcap\Phi$.
\end{proof}

Since the closed filters of $M^*$ are the same in the filter-topology and the weak topology, $(M^*)^*$ is isomorphic as an $\mathbb{F}$-module whether $M^*$ is endowed with one or the other. Since there is an isomorphism $M\cong(M^*)^*$ for complete modules $M$, the weak topology can be defined for them as well.

\begin{proposition}
Every complete module $M$ with a topology with respect to filters has a refined topology with respect to the order, referred to as its \textbf{weak topology}, and the order reversing bijection $\overline{M}\to \overline{M^*}$ is continuous.
\end{proposition}

\begin{proposition}
Consider two $\mathbb{F}$-modules $M_1$ and $M_2$, with either the discrete topology, or a topology with respect to filters. Let us define the closed filters of $M_1\otimes M_2$ to be those filters $F$ where $\{a\mid a\otimes b\in F\}$ and $\{b\mid a\otimes b\in F\}$ are closed. Then $(M_1\otimes M_2)^*\cong\Hom(M_1,M_2^*)$, where $\Hom(M_1,M_2^*)$ is the module of continuous homomorphisms to $M_2$ with the topology with respect to filters.
\end{proposition}

\begin{proof}
Elements of $(M_1\otimes M_2)^*$ are closed filters on $M_1\otimes M_2$, while elements of $\Hom(M_1,M_2^*)$ are continuous maps from $M_1$ to closed filters of $M_2$. Given a filter $F$ on $M_1\otimes M_2$ and a map $\varphi\colon M_1\to M_2^*$, we will identify them if for any $m_1\in M_1$ and $m_2\in M_2$, $m_1\otimes m_2\in F$ if and only if $\varphi(m_1)\ni m_2$. It can be checked that this defines a bijection between filters on $M_1\otimes M_2$ and maps $M_1\to M_2^*$.

Now assume that there is corresponding pair of a filter $F$ and a map $\varphi$. Given an element $m_1\in M_1$, we need to see when $\varphi(m_1)$ is closed. It consists of those $m_2$ where $m_1\otimes m_2\in F$, which is a closed set if $F$ is closed. To make $\varphi$ continuous, let us fix a closed set from the basis of topology of $M_2^*$, $C_{m_2}$ for some $m_2\in M_2$. Then $\varphi^{-1}(C_{m_2})=\{m_1\mid\varphi(m_1)\in C_{m_2}\}$. Since $\Phi\in C_{m_2}$ if and only if $m_2\in\Phi$, we get $\varphi^{-1}(C_{m_2})=\{m_1\mid m_2\in\varphi(m_1)\}=\{m_1\mid m_1\otimes m_2\in F\}$, which is closed if $F$ is closed. Furthermore, if $\varphi$ maps pointwise to closed filters and is continuous, then $F$ is closed as well.
\end{proof}

\begin{proposition}\label{prop:proddual}
$(A+B)^*\cong A^*\times B^*$, $(A\times B)^*\cong A^*+B^*$.
\end{proposition}

\begin{proof}
Since $A\to A^*$ defines a contravariant functor from the category of $\mathbb{F}$-modules to itself, and the product and coproduct are dual to each other, these equalities hold.
\end{proof}

\smallskip

\begin{definition}
Given a module $M$ with a topology, we say that \textbf{all sums exist} if for any set $S\subseteq M$ there is a lower bound.
\end{definition}

In particular, finite modules (with the discrete topology) and complete modules with the weak topology are such that all sums exist.

\begin{proposition}
Assume that all sums exist in $M_1$. A homomorphism $\varphi\colon M_1\to M_2$ admits a natural dual $\varphi^*\colon M_2^*\to M_1^*$, identified by $\varphi^*(\mu)=\left(\sum\{m\in M_1\mid \varphi(m)\ge \mu^*\}\right)^*$ where $\mu^*$ is defined through $(\mu^*)^*=\mu$.
\end{proposition}

\begin{proof}
There is a natural map $\varphi^*\colon M_2^*\to M_1^*$ defined as $\varphi^*(\mu)(m)=\mu(\varphi(m))$, so we just have to prove that it is indeed given by the above identification. Consider the $\varphi^*$ defined as in the statement, and denote $\mu_0:=\sum\{x\in M_1\mid \varphi(x)\ge \mu^*\}$. We have $\varphi(x)\ge\mu^*$ if and only if $\mu(\varphi(x))=1$. We need to prove that $\mu_0^*(x)=\mu(\varphi(x))$ for all $x\in M_1$.

Consider a $\varepsilon\in\mathbb{F}^\times$. The set $\{x\in M_1\mid \varphi(x)\ge \mu^*\}$ is in fact a filter generated by $\mu_0$, since $\varphi$ is order preserving, hence $\mu_0^*(x)=\varepsilon$ if and only if $\varepsilon^{-1}x\ge\mu_0$. Since $\mu_0$ is the sum of all elements in $F_{\mu_0}$, this is equivalent to $\varphi(\varepsilon^{-1}x)\ge\mu^*$, and by the definition of $\mu^*$, this is $\mu(\varphi(x))=\varepsilon$. Since this is an equivalence, this also entails that $\mu^*(x)=0$ if and only if $\mu(\varphi(x))=0$.
\end{proof}

Given two finite modules, $M_1$ and $M_2$, a homomorphism is certainly determined if the image of generators of $M_1$ are given. However, not all such maps on the generators extend to the whole $M_1$.

\begin{lemma}
Let $M_1$ and $M_2$ be modules, and $G_1$ be a set of generators of $M_1$ closed under multiplication, and $G_2$ a set of generators of $M_2^*$ closed under multiplication. There is a bijection between homomorphisms $\varphi\colon M_1\to M_2$ and pairs of operation-preserving maps $u\colon G_1\to M_2$ and $v\colon G_2\to M_1^*$ such that $u(g)\ge\gamma$ in $M_2$ if and only if $g\ge v(\gamma)$ in $M_1$ for $g\in G_1$ and $\gamma\in G_2$.
\end{lemma}

\begin{proof}
If there is a homomorphism $\varphi$, then clearly $u:=\varphi|_{G_1}$ and $v:=\varphi^*|_{G_2}$ satisfy the condition. Conversely, a map $u\colon G_1\to M_2$ extends to a homomorphism if and only if for every pair of sums $\sum_i g_i=\sum_i g_i'$ with $g_i$ and $g_i'\in G_1$, we have $\sum_i u(g_i)=\sum_i u(g_i')$. Consider such a pair, and denote $A:=\sum_i u(g_i)$ and $B:=\sum_i u(g_i')$. If $A\ne B$, there is at least a single $\gamma\in G_2$ such that $A\ge\gamma$ but $B\not\ge\gamma$ in $M_2$, or vice versa. Since $u(g)\ge\gamma$ for some $g\in G_2$ if and only if $g\ge v(\gamma)$, clearly $g_i\ge v(\gamma)$ and $g_i'\not\ge v(\gamma)$. However, by our assumption, $v(\gamma)\le\sum_i g_i=\sum_i g_i'\not\ge v(\gamma)$, a contradiction.
\end{proof}

The following theorem shows which homomorphisms exist.

\begin{theorem}
\label{thm:homstruct}
Let $M_1$ and $M_2$ be modules, either finite with the discrete topology, or complete with the weak topology. Let $G_1$ be a set of generators of $M_1$ closed under multiplication, and $G_2$ a set of generators of $M_2^*$. Given a map $f\colon G_1\to M_2$ that preserves operations, this extends to a homomorphism from $M_1$ if and only if for each $\gamma\in G_2$, $F_\gamma:=f^{-1}(\{x\in M_2\mid x\ge\gamma\})$ is such that for any finite subset $S\subseteq F_\gamma$ and $g\in G_1$, if $g\ge\sum_{s\in S}s$ then $g\in F_\gamma$.
\end{theorem}

\begin{proof}
Clearly if such a map $\varphi$ exists, $\varphi^{-1}(\{x\in M_2\mid x\ge\gamma\})$ is closed under addition, and all its elements are generated by $F_\gamma$. Conversely, let us construct a map $v\colon G_2\to M_1^*$ by defining $v(\gamma)=\sum \varphi^{-1}(\{x\in M_2\mid x\ge\gamma\})$. If $M_1$ is finite or complete with the weak topology, such a sum exists. Also, if $f(g)\ge\gamma$, then $g\in\varphi^{-1}(\{x\in M_2\mid x\ge\gamma\})$, hence $g\ge v(\gamma)$, satisfying the conditions of the previous lemma. Hence a homomorphism exists.
\end{proof}

\section{Congruences and ideals}

\subsection{Congruences}

A congruence in an $\Finfty$-module $M$ or $\Finfty$-algebra $A$ is an equivalence relation compatible with the natural algebraic structure. 

\begin{Definition}
A congruence $C$ in an $\Finfty$-module $M$ is a set of pairs $(a,b)$ where $a$, $b\in M$ so that
\begin{itemize}
\item for every $a\in M$, $(a,a)\in C$,

\item if $(a,b)\in C$ and $(b,c)\in C$, then $(a,c)\in C$,

\item if $(a,b)\in C$, then $(b,a)\in C$, and finally

\item if $(a,b)\in C$, then for every $c\in M$ we have that $(a\plus c,b\plus c)\in C$.
\end{itemize}
A congruence $C$ in an $\Finfty$-algebra $A$ is a congruence on the underlying module that furthermore satisfies
\begin{itemize}
\item if $(a,b)\in C$, then for every $c\in A$ we have that $(ac,bc)\in C$.
\end{itemize}
\end{Definition}

Clearly, the smallest congruence, $\Delta$, is the set of diagonal pairs, $\Delta:=\{(a,a)|a\in A\}$, for either modules or algebras. The maximal congruence is the set of all pairs. Moreover, notice that if $C$ is a congruence of an $\Finfty$-module $M$, then $M/C$ is also an $\Finfty$-module. Likewise, if $C$ is a congruence of an $\Finfty$-algebra, $A/C$ is also an $\Finfty$-algebra.

Restricting our study to algebras, annihilators of elements of $A$ give rise to congruences.

\begin{Lemma}
Let $a\in A$ and $C$ a congruence in $A$. Then the set of pairs
\[Ann_C(a)=\{(b,c)|(ab,ac)\in C\}\]
is a congruence.
\end{Lemma}

\begin{proof}
We leave the proof of this statement to the reader.
\end{proof}

\subsection{Ideals}

Since modules are partially ordered sets, we may define their ideals and filters. Note that ideals will be defined differently for modules and algebras. Let us fix an $\mathbb{F}$-module $M$.

\begin{definition}
An \textbf{ideal of a module}, $I$ is such that $I\plus M\subseteq I$ and $\mathbb{F}\cdot I\subseteq I$ and $0\in I$.
A \textbf{filter of a module}, $F$ is such that if $a\plus b\in F$ then $a\in F$ and $b\in F$, and also $0\not\in F$.
\end{definition}

\begin{definition}
The \textbf{ideal} (or \textbf{kernel}) of a congruence $C$ is the equivalence class of $0$. The \textbf{maximal congruence} for an ideal $I$, if it exists, is the maximal congruence whose ideal is $I$.
\end{definition}

\begin{definition}
A \textbf{maximal filter with respect to an ideal $I$} is a maximal filter among filters that do not intersect $I$. A \textbf{maximal filter} is one that is maximal with respect to the trivial ideal $\{0\}$. A module is \textbf{separable with respect to the order} or just \textbf{separable} if for any pair of distinct elements $a$, $b\in M$, there is a maximal filter $F$ such that $a\in F$ and $b\not\in F$, or vice versa.
\end{definition}

\begin{lemma}
For an ideal $I$ and an element $x\not \in I$, there is a maximal filter with respect to $I$ that does not contain $x$.
\end{lemma}

\begin{proof}
Zorn's lemma.
\end{proof}

\begin{theorem}
\label{thm:maxcongmodule}
In a module $M$, every ideal has a corresponding maximal congruence $C$, characterized by the property that $M/C$ is separable.
\end{theorem}

\begin{proof}
Every ideal $I$ has a corresponding minimal congruence such that $a\sim b$ if and only if $a=b$ or $a$, $b\in I$. Therefore, by passing to $M/I$, it is enough to check the statement for $I=\{0\}$.

Let us denote by ${\cal F}(a)$ the set of maximal filters containing $a$. Let us define the equivalence relation $C$ as $a\sim b$ if and only if ${\cal F}(a)={\cal F}(b)$. This is a congruence, since ${\cal F}(\lambda a)=\{\lambda F\mid F\in{\cal F}(a)\}$, and ${\cal F}(a\plus b)={\cal F}(a)\cap{\cal F}(b)$. Furthermore, $M/C$ is separable.

We only need to show that this is in fact the maximal congruence. We may pass to the module $M/C$ via the assumption $M/C=M$, meaning that $M$ is separable. Assume that there is a non-trivial congruence $C$ whose ideal is trivial, meaning that $a\sim b$ for some distinct pair of elements. Since $M$ is separable, we have a maximal filter $F$ such that, for instance, $a\in F$ and $b\not\in F$. Since $F$ is maximal, the filter generated by $F$ and $b$ contains $0$, that is $x\plus b\le 0$ for some $x\in F$. Under the congruence $C$, we get $0=x\plus b\sim x\plus a$. Since $x$, $a\in F$, and $F$ is a filter, $x\plus a\in F$ as well, and so $x\plus a\ne 0$. Then the ideal of $C$ contains a non-zero element $x\plus a$, contradicting our assumption.
\end{proof}

\smallskip

A similar result can be achieved for algebras. Let us fix an algebra $A$.

\begin{definition}
An \textbf{ideal of an algebra}, $I$ is such that it is an ideal of the module, furthermore $A\cdot I\subseteq I$. A maximal filter with respect to an ideal is a maximal filter of the underlying module. A \textbf{quasimaximal filter with respect to an ideal $I$}, $\Phi$ is such that there is a maximal filter $F$ and an $a\in A$ (possibly $a=1$) such that $\Phi=\{x\in A\mid ax\in F\}$. A \textbf{quasimaximal filter} is one that is quasimaximal with respect to the trivial ideal $\{0\}$. An algebra is \textbf{quasiseparable} if for any pair of distinct elements $a$, $b\in M$, there is a quasimaximal filter $F$ such that $a\in F$ and $b\not\in F$, or vice versa.
\end{definition}

We shall denote the set $\{x\in A\mid ax\in F\}$ by $F:a$.

\begin{theorem}
\label{thm:maxcongalgebra}
In an algebra $A$, every ideal has a corresponding maximal congruence $C$, characterized by the property that $M/C$ is quasiseparable.
\end{theorem}

\begin{proof}
The proof is similar to the case of modules. Every ideal $I$ has a corresponding minimal congruence, therefore, by passing to $A/I$, it is enough to check the statement for $I=\{0\}$.

Let us denote by ${\cal F}(a)$ the set of quasimaximal filters containing $a$. Let us define the equivalence relation $C$ as $a\sim b$ if and only if ${\cal F}(a)={\cal F}(b)$. Clearly ${\cal F}(\lambda a)=\{\lambda F\mid F\in{\cal F}(a)\}$, ${\cal F}(a\plus b)={\cal F}(a)\cap{\cal F}(b)$. Furthermore if ${\cal F}(a)={\cal F}(b)$, we need to prove ${\cal F}(ac)={\cal F}(bc)$. The antecendent means that $ua\in F$ if and only if $ub\in F$ for all $F$ maximal filters and $u\in A$. In particular, this holds for $u=vc$ for any $v\in A$, hence $v(ac)\in F$ if and only if $v(bc)\in F$. This determines that ${\cal F}(ac)={\cal F}(bc)$. Also, $A/C$ is quasiseparable.

To show that this is in fact the maximal congruence, we pass to the algebra $A/C$. Let us assume that $A$ is quasiseparable and $C$ is the trivial congruence. Assume that there is a non-trivial congruence $C$ whose ideal is trivial, meaning that $a\sim b$ for some distinct pair of elements. Since $A$ is quasiseparable, we have a maximal filter $F$ and $u\in A$ such that, for instance, $ua\in F$ and $ub\not\in F$. Since $F$ is maximal, the filter generated by $F$ and $ub$ contains $0$, that is $x\plus ub=0$ for some $x\in F$. Under the congruence $C$, we get $0=x\plus ub\sim x\plus ua\in F$, and $x\plus ua\ne 0$. Then the ideal of $C$ contains a non-zero element $x\plus ua$, contradicting our assumption.
\end{proof}

\subsection{Congruences in $\Finfty$-fields}

In this section we investigate fields over $\Finfty$ and we prove some elementary statements which are needed for our proof of the prime decomposition. Recall from \ref{def:algebra} the following.

\begin{Definition}
We say that an $\Finfty$-algebra is a field, if every $a\ne 0$ has a multiplicative inverse. 
\end{Definition}

Fields over $\Finfty$ can have non-trivial congruences, on the other hand, the kernels of these congruences are always trivial.

\begin{Lemma}
\label{lemma:ker0}
Let $F$ be a field over $\Finfty$, and $C$ a proper congruence. Then the kernel of $C$ is trivial.
\end{Lemma}

\begin{proof}
Assume that $(a,0)\in C$ for some $a\ne 0$. Then, $a$ is a unit, hence $(1,0)\in C$ implying that $C$ cannot be proper.
\end{proof}

Therefore, by Theorem \ref{thm:maxcongalgebra}, the field has to have a unique maximal (proper) congruence. We construct this unique maximal congruence.

\begin{Proposition}
Let $F$ be a field over $\Finfty$. Then, the set $C=\{(a,b)|a\ne 0, b\ne 0,a\plus b\ne 0\}\cup \{(0,0)\}$ is a congruence.
\end{Proposition}

\begin{proof}
We begin with proving transitivity. Assume that $(a,b)\in C$ and $(b,c)\in C$, we prove that $(a,c)\in C$. It is enough to show that whenever $a\plus b\ne 0$ and $a\plus c\ne 0$, then $b\plus c\ne 0$. Indeed consider 
\[b(1\plus ab^{-1})(1\plus ca^{-1})=b\plus c\plus ...\]
and since every element on the left hand side is a unit, thus the right hand side cannot be $0$.

Next we show that if $(a,b)\in C$ and $(c,d)\in C$, then $(a\plus c,b\plus d)\in C$. Indeed, it is enough to show that whenever $a\plus b\ne 0$ and $c\plus d\ne 0$, then either $a\plus b\plus c\plus d\ne 0$ or $a\plus c=b\plus d=0$. Consider the following product
\[(a\plus c)(1\plus ba^{-1})(1\plus dc^{-1})=a\plus b\plus c\plus d\plus ...\]
We see that if $a\plus b\plus c\plus d=0$, then $a\plus c=0$.

Finally, we show that if $(a,b)\in C$ for $a\plus b\ne 0$ and $c\in F$ then $(ac,bc)\in C$. Clearly either $ac\plus bc=c(a\plus b)\ne 0$ or $c=0$ and in this case $(ac,bc)=(0,0)\in C$.
\end{proof}

The above congruence is indeed maximal, since if $(a,b)\in C$ for some $a\ne 0$ and $a\plus b=0$, then $(0,b)=(a\plus b,b)\in C$ and by Lemma \ref{lemma:ker0} it cannot be proper.

Now, we characterize all congruences. Notice that a congruence $C$ of a field $F$ can be characterized by the equivalence class of 1. 

\begin{Proposition}
Let $C$ be a congruence. Then, if $x$ and $y$ are in the equivalence class of 1, then so are $xy^{-1}$, $xy$, $\lambda x\plus \mu y$ for every $\lambda, \mu\in F$ satisfying $\lambda\plus \mu=1$.
\end{Proposition}

\begin{proof}
The first two assertions are trivial. We prove the third one. Since $(x,1)\in C$ and $(y,1)\in C$, therefore $(\lambda x,\lambda)$ and $(\mu y,\mu)$ are in $C$, and thus so is $(\lambda x\plus \mu y,1)$.
\end{proof}

Actually this completely characterizes a congruence.

\begin{Proposition}
\label{prop:S}
Let $S$ be a subset of $F\setminus 0$ so that whenever $x,y\in S$, then $xy^{-1}$, $xy$ and $\lambda x\plus \mu y$ are in $S$ as well for every $\lambda,\mu\in F$ so that $\lambda\plus \mu=1$. Then, the set 
\[C=\{(a,b)|a\ne 0, b\ne 0, ab^{-1}\in S\}\cup \{(0,0)\}\]
is a congruence.
\end{Proposition}

\begin{proof}
The only assertion which is not trivial is that if $(a,b)\in C$ and $(c,d)\in C$ then $(a\plus c,b\plus d)\in C$. If $b\plus d\ne 0$, then $b(b\plus d)^{-1}\plus d(b\plus d)^{-1}=1$, and hence $ab^{-1}b(b\plus d)^{-1}\plus cd^{-1}d(b\plus d)^{-1}=(a\plus c)(b\plus d)^{-1}\in S$, and hence $(a\plus c,b\plus d)\in C$. Otherwise, if $b\plus d=0$, then by symmetry we get that $a\plus c=0$ and we are done.
\end{proof}

An easy corollary of the above characterization is the following.

\begin{Corollary}
\label{cor:gen}
Let $x\ne 0$, then the equivalence class of 1 in congruence generated by $(x,1)$ is the set
\[\left\{\frac{\sum_{i=1}^n \lambda_i x^i}{\sum_{j=1}^k \mu_j x^j}:\sum_{i=1}^n \lambda_i=\sum_{j=1}^k \mu_j=1\right\}\]
\end{Corollary}

\begin{proof}
We see that these elements have to be in the equivalence class and we also see that this set is closed under the operations listed in Proposition \ref{prop:S}.
\end{proof}

\section{Prime congruences}

In this section we define prime congruences and we prove some simple statements about them. We also explicitly compute all prime congruences of $\Finfty[x]$. We begin with the motivation.

In the work of D\'aniel Jo\'o and Kalina Mincheva (\cite{joo2014}), prime congruences were defined in additively idempotent semirings in a very straightforward manner. If the semiring were a ring, and $C$ a congruence in it, $(a,b)\in C$ would hold if and only if $(a-b,0)\in C$. For any pairs $(a-b,0)$ and $(c-d,0)$, their product $((a-b)(c-d),0)\in C$ if and only if $(ac\plus bd,ad\plus bc)\in C$. Therefore they defined $C$ to be a prime if $(ac+bd,ad+bc)\in C$ entails that either $(a,b)\in C$ or $(c,d)\in C$, and this definition holds in semirings in general.

Unfortunately, for our $\Finfty$-algebras, this condition is too strong, since choosing $c=0$, $ac\plus bd=ad\plus bc=0$, and $(0,0)\in C$, hence all $(a,b)\in C$. For intuition, we turned to the ring $\mathbb{Z}_{\infplace}$. In Nikolai Durov's work (\cite{durov2007}), $\mathbb{Z}_{(\infplace)}$ is isomorphic to the closed interval $[-1,1]$, and instead of addition, we have convex combinations, such as $\frac{a+b}{2}$. To avoid the absorbing properties of $0$, let us interpret the condition $(a,b)\in C$ for some congruence as the harmonic difference $a\star b:=\frac{1}{\frac{1}{a}-\frac{1}{b}}$, instead of the standard difference $a-b$. Then $(ac\plus bd)(ad\plus bc)(a\star b)(c\star d)=abcd((ac\plus bd)\star(ad\plus bc))$. This motivates the following preliminary definiton for a prime congruence.

\begin{Definition}
We say that a proper congruence $C$ is prime if the following two conditions hold
\begin{enumerate} 
\item Whenever $abcd(ac\plus bd,ad\plus bc)\in C$ then either
\begin{itemize}
\item $(a,b)\in C$ or
\item $(c,d)\in C$ or
\item $(ac\plus bd,0)\in C$ or
\item $(ad\plus bc,0)\in C$.
\end{itemize}
\item Whenever $abc(ac,bc)\in C$ then either
\begin{itemize}
\item $(a,b)\in C$ or
\item $(ac,0)\in C$ or
\item $(bc,0)\in C$.
\end{itemize}
\end{enumerate}
\end{Definition}

\begin{remark}
The reason that there are two conditions is that the direct sum contains other elements than pairs of elements.
\end{remark}

Once prime congruences are specified, we can define $\Spec$ of any $\Finfty$-algebra (as a topological space) as follows:

\begin{definition}
Given a congruence $C$ in $\Finfty$-algebra $A$ we denote the set of all prime congruences containing $C$ by $V(C)$. Let $\Spec A$ be the \textbf{spectrum of $A$}, consisting of prime congruences. The closed sets of $\Spec A$ are generated by those of the form $V(C)$ where $C$ is a congruence.
\end{definition}

\begin{Lemma}
Let $C$ be a prime congruence. Then $(ab,0)\in C$ implies that $(a,0)\in C$ or $(b,0)\in C$.
\end{Lemma}

\begin{proof}
Assume first that $(x^2,0)\in C$ holds for a prime congruence $C$. Applying the second condition for $a=1$, $b=-1$ and $c=x$ that either $(x,0)$ or $(-x,0)$ holds in $C$.

Now, assume that $(ab,0)\in C$ holds for a prime congruence $C$. Applying the second condition for $c=1$, we get that either $(a,b)\in C$, or $(a,0)$ or $(b,0)\in C$. If $(a,b)\in C$, then $(a^2,0)\in C$, therefore $(a,0)\in C$.
\end{proof}

\begin{Lemma}
Let $C$ be a prime congruence. Assume that neither $(ac,0)$ nor $(bc,0)$ holds in $C$. Then $(ac,bc)\in C$ implies that $(a,b)\in C$.
\end{Lemma}

\begin{proof}
Trivial.
\end{proof}

The above lemmas show that we can simplify our notion of a prime congruence to the following equivalent definition.

\begin{Definition}
We say that a proper congruence $C$ is prime if the following two conditions hold
\begin{enumerate} 
\item Whenever $(ac\plus bd,ad\plus bc)\in C$ then either
\begin{itemize}
\item $(a,b)\in C$ or
\item $(c,d)\in C$ or
\item $(ac\plus bd,0)\in C$ or
\item $(ad\plus bc,0)\in C$.
\end{itemize}
\item Whenever $(ac,bc)\in C$ then either
\begin{itemize}
\item $(a,b)\in C$ or
\item $(c,0)\in C$.
\end{itemize}
\end{enumerate}
\end{Definition}

From now on, we use this definition for a prime congruence. Instead of writing $(ac\plus bd,ad\plus bc)$ we will write $(a,b)(c,d)$. We proceed with some simple lemmas needed to characterize the prime congruences of $\Finfty[x]$.

\begin{Lemma}
Let $C$ be a prime congruence. Then for any two elements $a,b\in A$ we have that $a\geq b$ or $a\leq b$ or $a\plus b=0$ in $A/C$.
\end{Lemma}

\begin{proof}
We can assume that neither $a$ nor $b$ is identified with $0$ in $A/C$. Consider the following identity
\[(a\plus b,a)(a\plus b,b)=((a\plus b)^2,(a\plus b)^2).\]
From the first condition, we obtain that either $((a\plus b)^2,0)$, $(a\plus b,a)$ or $(a\plus b,b)$ is in $P$ for all $a,b\in A$.
\end{proof}

As a consequence we see that if $C$ is a prime congruence, then $A/C$ is a union of totally ordered chains with sums of elements in different chains being $0$. 

\begin{Lemma}
Let $C$ be a prime congruence in an $\F_\infty$-algebra $A$ so that in $A/C$ we have that $a>b$ and $c>d$. Then either $ac>bd$ or $(ac,0)\in C$.
\end{Lemma}

\begin{proof}
The statements $ac\geq ad\geq bd$ and $ac\geq bd\geq bd$ hold for any congruence. On the other hand if $ac=bd$, then $ac=ad=bc=bd$ has to hold as well, and hence $(a,b)(c,d)=(ac\plus bd,ad\plus bc)\in C$. The latter implies that $(ac,0)\in C$.
\end{proof}

Moreover we can take roots in prime congruences.

\begin{Lemma}
Let $C$ be a prime congruence of an $\Finfty$-algebra $A$. Assume that $(a^n,b^n)\in C$ for some $(a,b)\not \in C$. Then $(a\plus b,0)\in C$.
\end{Lemma}

\begin{proof}
Since $C$ is a prime congruence, therefore in $A/C$, we have $a>b$ or $b>a$ or $a\plus b=0$. The first two cases cannot hold, since $a^n=b^n$.
\end{proof}

Now, we characterize the prime congruences of $\Finfty[x]$. Recall that elements of $\Finfty[x]$ are $0$ and polynomials of the form $\sum_{i\in I} \lambda_{i} x^{i}$ where $\lambda_i=\pm 1$ and $I\subseteq\mathbb{Z}$ is finite. We begin with a simple lemma.

\begin{Lemma}\label{lem:1+x}
Let $A$ be an $\Finfty$-algebra and assume that $(1\plus x^n,1)\in C$ and $(1\plus x^m,1)\in C$ hold for a congruence $C$ for some $n>m$. Then $(1\plus x^{n\plus m},1)\in C$. Moreover if $C$ is prime, then $(1\plus x^{n-m},1)$ is also in $C$.
\end{Lemma}

\begin{proof}
Since $(1\plus x^n,1)\in C$ and $(1\plus x^m,1)\in C$, therefore
\[(1\plus x^n\plus x^m\plus x^{n\plus m},1)\]
also holds in $C$ which implies that $(1\plus x^{n\plus m},1)\in C$. 

Now, assume that $C$ is prime. Consider $1$ and $x^{n-m}$. Since $C$ is prime, one of the following holds in $C$:
\begin{itemize}
\item $(1\plus x^{n-m},0)$: In this case we get that $(x^m\plus x^n,0)$ holds, but this cannot be true, since $1\plus x^n=1\plus x^m=1$ in $A/C$.
\item $(1\plus x^{n-m},x^{n-m})$: In this case $1\plus x^{n-m}\plus x^{2(n-m)}\plus ...\plus x^{m(n-m)}=x^{m(n-m)}$ and also it equals to $1\plus x^{m(n-m)}=1$ in $A/C$, therefore $(x^{m(n-m)},1)$ holds in $A/C$, meaning that $(x^{m-n},1)$ holds in $A/C$, and we are done.
\item $(1\plus x^{n-m},1)$: We are done.
\end{itemize}
\end{proof}

We are ready to compute the prime congruences of $\Finfty[x]$.

\begin{Theorem}
\label{thm:polyprimes}
The prime congruences of $\Finfty[x]$ are the following.
\begin{enumerate}
\item The congruence generated by $(1\plus x,x)$.
\item The congruence generated by $(x,1)$.
\item The congruence generated by $(x,0)$.
\item The congruence generated by $(1\plus x,1)$.
\item The congruence generated by $(-1\plus x,x)$.
\item The congruence generated by $(x,-1)$.
\item The congruence generated by $(-1\plus x,-1)$.
\item Every $n>0$, the prime congruence $P$ generated by $(1\pm x,0)$, $(1\pm x^2,0)$, ... $(1\pm x^{n-1},0)$ and $(1\plus x^n,x^n)$.
\item Every $n>0$, the prime congruence $P$ generated by $(1\pm x,0)$, $(1\pm x^2,0)$, ... $(1\pm x^{n-1},0)$ and $(1\plus x^n,1)$.
\item Every $n>0$, the prime congruence $P$ generated by $(1\pm x,0)$, $(1\pm x^2,0)$, ... $(1\pm x^{n-1},0)$ and $(-1\plus x^n,x^n)$.
\item Every $n>0$, the prime congruence $P$ generated by $(1\pm x,0)$, $(1\pm x^2,0)$, ... $(1\pm x^{n-1},0)$ and $(-1\plus x^n,-1)$.
\item Every $n>0$, the prime congruence $P$ generated by $(1\pm x,0)$, $(1\pm x^2,0)$, ... $(1\pm x^{n-1},0)$ and $(x^n,1)$.
\item Every $n>0$, the prime congruence $P$ generated by $(1\pm x,0)$, $(1\pm x^2,0)$, ... $(1\pm x^{n-1},0)$ and $(x^n,-1)$.
\item The prime congruence $P$ generated by $(1\pm x,0)$, $(1\pm x^2,0)$, ...
\end{enumerate}
\end{Theorem}

\begin{proof}
Let $P$ be a prime congruence. We have three cases:
\begin{enumerate}
\item $1\plus  x=x$ holds in $\Finfty[x]/P$
\item $1\plus x=1$ holds in $\Finfty[x]/P$
\item Neither of the above, in particular $1\plus x=0$ holds in $\Finfty[x]/P$.
\end{enumerate}

We investigate all cases:
\begin{enumerate}
\item $1\plus x=x$: Let $P$ be the smallest congruence so that $1\plus x=x$ holds in $\Finfty[x]/P$. In this case, we can replace any polynomial with its highest degree term in $\Finfty[x]/P$ if all coefficients are equal, otherwise $0$. We see that the corresponding smallest congruence is indeed prime, because degree is additive. 

Is there any prime congruence $Q$ containing $P$? If $Q$ is any other congruence, then either $(x^n,x^m)\in Q$ or $(x^n,0)\in Q$. (If $(x^n,-x^m)\in Q$, then $x^n-x^m=0$ implies that $(x^n,0)\in Q$) If $(x^n,x^m)\in Q$ then $(1,x^{m-n})\in Q$ (without loss of generality, we can assume $m>n$), but $x\geq 1$, so $(x,1)\in Q$ and in this case $\Finfty[x]/Q=\Finfty$. Similarly, if $(x^n,0)\in Q$, then $(x,0)\in Q$, and in this case $\Finfty[x]/Q=\Finfty$.

\item $1\plus x=1$: Let $P$ be the smallest congruence so that $1\plus x=1$ holds in $\Finfty[x]/P$. In this case, we can replace any polynomial with its smallest degree term in $\Finfty[x]/P$ if all coefficients are equal, otherwise $0$. We see that the corresponding smallest congruence is indeed prime, because degree is additive. 

Is there any prime congruence $Q$ containing $P$? If $Q$ is any other congruence, than either $(x^n,x^m)\in Q$ or $(x^n,0)\in Q$. If $(x^n,x^m)\in Q$ then $(1,x^{m-n})\in Q$ (without loss of generality, we can assume that $m>n$), but $x\leq 1$, so $(x,1)\in Q$ and in this case $\Finfty[x]/Q=\Finfty$. If $(x^n,0)\in Q$, then $(x,0)\in Q$, which contradicts $1\plus x=1$.

\item Last case: In this case, neither of the above holds. We can also assume that neither $-1\plus x=x$ nor $-1\plus x=-1$ holds because we can replace $x$ by $-x$ and we receive one of the cases above. Let $P$ be any prime congruence in this case. We have basically three cases: $(1\pm x^n,0)\in P$ for every $n$ or $(1\plus  x^n,1)\in P$ for some $n$ or $(1\plus x^n,x^n)\in P$ for some $n$ or $(-1\plus x^n,-1)\in P$ for some $n$ or $(-1\plus  x^n,x^n)\in P$ for some $n$. First, we assume that there is an expression $1\pm x^n$ which is not 0. Let $n$ be a smallest such $n$, and moreover assume that we have $(1\plus x^n,1)\in P$. Then, by Lemma \ref{lem:1+x}, the $m$'s satisfying $(1\plus x^m,1)\in P$ have to be divisible by this $n$. We see that the smallest congruence satisfying this condition is prime. Can a congruence $Q$ contain this congruence? We see that it can only happen if $x^n=1$ in that congruence. We leave to the reader to complete the cases when $(-1\pm x^n,-1)$ or $(\pm 1\plus x^n,x^n)$ is in $P$.

Finally, we have the case that $(1\pm x^n,0)$ holds for every $n$. We can see that the smallest such congruence $P$ is prime. Can there be any prime congruence $Q$ containing $P$? Since every polynomial containing at least 2 monomials is identified with $0$, hence the only possibility is that $x^n$ is identified with $x^m$. In that case we get that $x^n=0$ or $x^{n-m}=1$ (this cannot hold). Therefore $x=0$.
\end{enumerate}
\end{proof}

Geometrically, the prime congruences generated by $(x,\pm 1)$ correspond to evaluation at $x=\pm 1$, and the prime congruence generated by $(x,0)$ corresponds to evaluation at $x=0$. Furthermore the prime congruences listed in 12. and 13. are listed in Section \ref{sec:mod} in the Example part. These are finite field extensions of $\Finfty$.

Therefore, we obtain that the geometry of $\Spec \Finfty[x]$ is very similar to $\Spec \Z/p\Z[x]$, for instance the closed points of $\Spec \Finfty[x]$ correspond to elements of $\Finfty$ and some finite extensions of $\Finfty$.

\section{Krull dimension}
In this section we prove that the Krull dimension of a polynomial algebra over $\Finfty$ is the number of indeterminants.

We say that an $\Finfty$-algebra $A$ has Krull dimension $n$ if the longest chain of prime congruences has length $n+ 1$. We begin with an easy lemma.

\begin{Lemma}
Let $A$ be an $\Finfty$-algebra of Krull dimension $n$. Then the dimension of $A[x]$ is at least $n+ 1$. 
\end{Lemma}

\begin{proof}
Let $P$ be the minimal element of a maximal chain of prime congruences of $A$. Then $\dim A/P=\dim A$. Moreover $\dim A/P[x]\geq \dim A/P+ 1$, since the congruence generated by $(x,0)$ is prime. Therefore
\[\dim A[x]\geq \dim A/P[x]\geq \dim A/P+ 1=\dim A+ 1.\]
\end{proof}

The next theorem gives a criterion when the dimension of $A[x]$ exceeds $n+1$.

\begin{Theorem}
\label{thm:krull}
Let $A$ be an $\Finfty$-algebra of Krull dimension $n$. Then either
\begin{itemize}
\item $A[x]$ is of dimension $n+ 1$.
\item $A[x]$ is of dimension at least $n+ 2$, furthermore in this case there exist four distinct prime congruences in this chain  $P_1\subset P_2\subset P_3\subset P_4$ so that there exist $a,b\in A$ so that $(a,0)$, $(b,0)$ are in $P_3$ and for some $j$, $(ax^j,b)\in P_2\setminus P_1$.
\end{itemize}
\end{Theorem}

\begin{proof}
If the Krull dimension of $A[x]$ is at least $n+2$, then there exist prime congruences $P_1\subset P_2\subset P_3\subset P_4$ of $A[x]$ so that $P_1|_A=P_2|_A$ and $P_3|_A=P_4|_A$ (using the natural map $A\ra A[x]$) so that $P_1\ne P_2$ and $P_3\ne P_4$. (Note that we do not assume that $P_2$ and $P_3$ are distinct). 

Since each prime congruence contains $(a\plus b,a)$, $(a\plus b,b)$ or $(a\plus b,0)$ for any elements $a,b\in A[x]$, hence we can assume that for any $(p(x),q(x))\in P_2\setminus P_1$, we have that $p(x)$ and $q(x)$ are either monomials or one of them is $0$. On the other hand if $(p(x),0)\in P_2\setminus P_1$, then $p(x)$ can be replaced by a monomial in $P_1$ and hence in $P_2$ as well. Hence, we can always assume that for any $(p(x),q(x))\in P_2\setminus P_1$ we have that $p(x)$ and $q(x)$ are monomials (or one of them is $0$).

The same holds for any pair $(p(x),q(x))\in P_4\setminus P_3$. Notice that if $(ax^n,0)\in P_{i+1}\setminus P_i$ (for $i=1$ or $3$), then from the prime property we get that either $(a,0)\in P_{i+1}\setminus P_i$ (which contradicts our original assumptions) or $(x,0)\in P_{i+1}\setminus P_i$. We separate cases.

\begin{enumerate}
\item First, we assume that there are monomials so that we have that 
\[(ax^n,bx^m)\in P_2\setminus P_1\]
and
\[(cx^k,dx^l)\in P_4\setminus P_3\]
and furthermore $(x,0)$ does not hold in any of the congruences. Then, we can use the cancellation property and we obtain that (assuming that $n\geq m$ and $k\geq l$)
\[(ax^{n-m},b)\in P_2\setminus P_1\]
and
\[(cx^{k-l},d)\in P_4\setminus P_3.\]
We see that
\[(a^{k-l}c^{n-m}x^{(n-m)(k-l)},b^{k-l}c^{n-m})\quad \mbox{and}\quad (a^{k-l}c^{n-m}x^{(n-m)(k-l)},d^{n-m}a^{k-l})\]
are contained in $P_4$, therefore $(b^{k-l}c^{n-m},d^{n-m}a^{k-l})\in P_4$. Since $P_4$ and $P_3$ are the same once they are restricted to $A$, therefore 
\[(b^{k-l}c^{n-m},d^{n-m}a^{k-l})\in P_3\]
implying that
\[(b^{k-l}c^{n-m}x^{(n-m)(k-l)},d^{n-m}a^{k-l}x^{(n-m)(k-l)})\in P_3.\]
Furthermore, since $(ax^{n-m},b)\in P_2\subset P_3$ we obtain that
\[(b^{k-l}c^{n-m}x^{(n-m)(k-l)},d^{n-m}b^{k-l})\in P_3.\]
Now, assume that $(b,0)\not \in P_3$, then 
\[(c^{n-m}x^{(n-m)(k-l)},d^{n-m})\in P_3.\]
In this case, either $(cx^{k-l},d)\in P_3$ or $(cx^{k-l}\plus d,0)\in P_3$. The first possibility is clearly impossible. If $(cx^{k-l}\plus d,0)\in P_3$, then since $(c^{k-l},d)\in P_4$, we obtain $(d,0)\in P_4$. This implies that $(d,0)\in P_3$, so $(cx^{k-l},0)\in P_3$ which yields contradiction.

Therefore we have that $(b,0)\in P_3$, which implies that $(a,0)\in P_3$ as well.

\item If $(x,0)\in P_2\setminus P_1$, then clearly $(cx^k,dx^l)\in P_4\setminus P_3$ cannot hold.

\item If $(x,0)\in P_{4}\setminus P_3$, then $(ax^n,bx^m)\in P_2\setminus P_1$ implies that $(ax^{n-m},b)\in P_2\setminus P_1$, so $(b,0)\in P_4$ and thus in $P_3$ as well. This implies that $(a,0)\in P_3$ and hence the statement holds.
\end{enumerate}
\end{proof}

\begin{remark}
If for a monomial $p(x_1,...,x_n)$ of $\Finfty[x_1,...,x_n]$ and a prime congruence $P$ we have $(p(x_1,...,x_n),0)\in P$, then for one of the variables $x_i$ of $x_1, ... ,x_n$ we have $(x_i,0)\in P$.
\end{remark}

\begin{Corollary}\label{cor:krull}
The Krull dimension of $\Finfty[x_1,...,x_n]$ is exactly $n$.
\end{Corollary}

\begin{proof}
We proceed by induction. For $n=0$, we are done. Assume that $\Finfty[x_1,...,x_{n-1}]$ is of dimension $n-1$, then we would like to prove that $\Finfty[x_1,...,x_n]$ is of dimension $n$. We prove that any chain of prime congruences is of length $n+1$. Assume the contrary, hence, we have a chain $L$ of prime congruences of length at least $n+2$:

\[L: P_0\subseteq P_1\subseteq ... \subseteq P_{n+1}\subseteq...\]

We have two cases.
\begin{enumerate}
\item Assume that none of the $P_j$ of the the chain $L$ of prime congruence contains any of the $(x_i,0)$. Then, the second case of Theorem \ref{thm:krull} cannot hold because of the above remark implying the statement of the Corollary for $A=\Finfty[x_1,...,x_{n-1}]$ and $x=x_n$.
\item One of the prime congruences $P$ contains one of the $(x_i,0)$. Let $P_j$ be the minimal prime congruence of the chain $L$ which contains one of the $(x_i,0)$. We can assume that it is $(x_n,0)$. Let us denote $\Finfty[x_1,...,x_{n-1}]$ by $A$. Then for any $P_k$ and $P_l$ for $k,l\geq j$ we have that $P_k|_A=P_l|_A$ if and only if $P_k=P_l$. Furthermore the second condition of Theorem \ref{thm:krull} cannot hold for four prime congruences of index at most $j$ again by the previous remark. Hence the chain
\[P_0|_A\subseteq P_1|_A\subseteq ... \subseteq P_{n+1}|_A\subseteq...\]
consists of distinct prime congruences except $P_{j-1}|_A=P_j|_A$. By induction, we are done.
\end{enumerate}
\end{proof}

\section{Prime decomposition}

We begin with some definitions. 

\begin{Definition}
We say that a congruence $C$ is radical, if whenever $(a,b)(a,b)\in C$, then $(a,b)\in C$ or $(a^2\plus b^2,0)\in C$ or $(ab,0)\in C$.
\end{Definition}

\begin{Definition}
We say that a congruence $C$ is cancellative, if whenever $(ab,ac)\in C$, then $(a,0)\in C$ or $(b,c)\in C$.
\end{Definition}

Notice that if $C$ is prime, then $C$ is radical and cancellative. Moreover

\begin{Lemma}
Let $C$ be a cancellative congruence, then $C$ is radical.
\end{Lemma}

\begin{proof}
Assume that $(a,b)(a,b)=(a^2\plus b^2,ab)\in C$. Then
\[(a\plus b)(a\plus b,a)=((a\plus b)^2,a^2\plus ab).\]
Since $(a^2\plus b^2,ab)\in C$, therefore we obtain that $(a\plus b)(a\plus b,a)\in C$, which implies that either $(a\plus b,0)$ or $(a\plus b,a)\in C$. If the first statement holds then $(ab,0)\in C$. If $(a\plus b,a)\in C$ holds, then by a symmetric line of thoughts we obtain that $(a\plus b,b)\in C$, so $(a,b)\in C$.
\end{proof}

Note that in an $\Finfty$-field every congruence is cancellative, hence radical. Morevoer, annihilators with respect to radical congruences are always radical.

\begin{Lemma}
Let $C$ be a radical congruence of an $\F_\infty$-algebra $A$. Then, for every $c\in A$, $Ann_C(c)$ is radical.
\end{Lemma}

\begin{proof}
Assume that $(a,b)^2c\in C$. Then, clearly $((a,b)c)^2\in C$, therefore $(a,b)c\in C$.
\end{proof}

We can also define annihilators of pairs $(c,d)$ as $Ann_C(c,d):=\{(a,b)|(a,b)(c,d)\in C$. This set is basically never a congruence, on the other hand, it is radical in the sense that if $(a,b)^2\in Ann_C(c,d)$, then $(a,b)\in Ann_C(c,d)$ for every radical congruence $C$.

In cancellative congruences, we can take roots of elements as explained in the following lemma. 

\begin{Lemma}
\label{lemma:atechlemma}
Let $C$ be a cancellative congruence of an $\Finfty$-algebra $A$. Then, if $(a,b)(c,d)\in C$, and $(a\plus b,0)\not \in C$ and $(c\plus d,0)\not \in C$, then for every $n$, $(a^n,b^n)(c,d)\in C$.
\end{Lemma}

\begin{proof}
We prove the statement by induction. For $n=1$, it is trivial. For $n=2$, we consider $(a\plus b)(a^2,b^2)(c,d)$, or in other words,
\[(a^3c\plus ab^2d\plus a^2bc\plus b^3d,ab^2c\plus a^3d\plus b^3c\plus a^2bd).\]
Since $(ac\plus bd,ad\plus bc)\in C$, hence in $A/C$,
\[(a^3c\plus ab^2d\plus a^2bc\plus b^3d)=(a^3c\plus ab^2c\plus a^2bd\plus b^3d)=\]
\[=(a^3c\plus ab^2d\plus b^3c\plus a^2bd)=(a^3d\plus ab^2d\plus b^3c\plus a^2bc)=\]
\[=(ab^2c\plus a^3d\plus b^3c\plus a^2bd).\]
Since $C$ is cancellative, we obtain that $(a^2,b^2)(c,d)\in C$.

We assume that the statement is true up to $n-1$, and we would like to prove it for $n$. Consider $(a^{n-1}\plus b^{n-1})(a^n,b^n)(c,d)$, or in other words, \[(a^{2n-1}c\plus a^{n-1}b^nd\plus a^nb^{n-1}c\plus b^{2n-1}d,a^{2n-1}d\plus a^{n-1}b^nc\plus a^nb^{n-1}d\plus b^{2n-1}c).\]
Since $(ac\plus bd,ad\plus bc)\in C$ and by the inductive hypothesis
\[(a^{n-1}c\plus b^{n-1}d,a^{n-1}d\plus b^{n-1}c)\in C,\]
we have that in $A/C$, the following equations hold
\[(a^{2n-1}c\plus a^{n-1}b^nd\plus a^nb^{n-1}c\plus b^{2n-1}d)=\]
\[=(a^{2n-1}c\plus a^{n}b^{n-1}d\plus a^{n-1}b^{n}c\plus b^{2n-1}d)=\]
\[=(a^{2n-1}c\plus a^{n}b^{n-1}d\plus a^{n-1}b^{n}d\plus b^{2n-1}c)=\]
\[=(a^{2n-1}d\plus a^{n}b^{n-1}c\plus a^{n-1}b^{n}d\plus b^{2n-1}c)=\]
\[=(a^{2n-1}c\plus a^{n-1}b^nc\plus a^nb^{n-1}d\plus b^{2n-1}c),\]
and we are done.
\end{proof}

As a Corollary we obtain the following surprising statement.

\begin{Corollary}\label{cor:lk}
Let $F$ be a field over $\Finfty$, and assume that $(a,b)(c,d)\in \Delta$ for some elements such that $a\plus b\ne 0$ and $c\plus d\ne 0$. Then any pair $(x,y)$ of the congruence generated by $(a,b)$ is annihilated by $(c,d)$, in other words, $(x,y)(c,d)\in \Delta$.
\end{Corollary}

\begin{proof}
By Corollary \ref{cor:gen} the congruence generated by $(a,b)$ consists of pairs of the form
\[\left(u(\sum \lambda_i (ab^{-1})^i),u(\sum \mu_j (ab^{-1})^j)\right)\]
for some $u\in F$, and some coefficients $\lambda_i$, $\mu_j$ so that $\sum \lambda_i=\sum \mu_j=1$. Therefore every pair is of the form
\[\left(u(\sum \lambda_i(\sum \mu_j)(ab^{-1})^i,u(\sum \mu_j(\sum \lambda_i)(ab^{-1})^j)\right).\]
Hence, after foiling out we see that any pair in the congruence generated by $(a,b)$ can be written as a sum of 
\[c_i(a^i,b^i)\]
for some $c_i\in F$. By the previous statement, our Corollary follows.
\end{proof}

The following proposition enables us to work with fields instead of cancellative algebras.

\begin{Proposition}\label{prop:frac}
Let $A$ be an $\Finfty$-algebra, then if $C$ is a cancellative congruence, then $A/C$ embeds into a field. 
\end{Proposition}

\begin{proof}
We leave it to the reader that the standard construction of a fraction field works here.
\end{proof}

One big advantage of working with fields is that we can more easily construct prime congruences.

\begin{Lemma}
Let $F$ be an $\Finfty$-field and $1\ne x\in F$. Then, a maximal congruence among all congruences which does not contain $(x,1)$ is a prime congruence.
\end{Lemma}

\begin{proof}
First of all, to make sense of the statement, we see that $\Delta$ is a congruence not containing $(x,1)$, and hence, by Zorn lemma, there is a maximal congruence, and we denote it by $P$. 

If $1\plus x=0$, then $(x,1)\in C$ implies that $(0,1)\in C$. Since a field has a unique maximal congruence which is prime, we are done in this case.

We assume that $1\plus x\ne 0$. We show that $P$ is prime. Suppose the contrary. One of the two assertions of being a prime has to fail. Assume that there exists $a,b,c,d\in F\setminus \{0\}$ such that $(a,b)(c,d)\in P$, but $(a,b)\not \in P$, $(b,c)\not \in P$, $(ac\plus bd,0)\not \in P$ and $(ad\plus bc,0)\not \in P$. The latter two conditions imply that $(a\plus b,0)\not \in P$ and $(c\plus d,0)\not \in P$. We leave it to the reader to show that the second assertion of being a prime cannot fail.

For simplicity, we look at $F/P$. In this field, any non-trivial congruence contains $(x,1)$. Consider the congruence generated by $(ab^{-1},1)$. By Corollary \ref{cor:lk}, any $(u,v)$ in the congruence generated by $(ab^{-1},1)$ is annihilated by $(c,d)$. So we obtain that $(x,1)(c,d)\in \Delta$ in $F/P$. By a similar argument, using that $1\plus x\ne 0$, we obtain that $(x,1)^2\in \Delta$ which is a contradiction because every congruence in an $\Finfty$-field is radical.
\end{proof}

We are ready to prove our main theorems of this section.

\begin{Theorem}\label{thm:field}
Let $F$ be an $\Finfty$-field. Then the trivial congruence $\Delta$ is the intersection of prime congruences.
\end{Theorem}

\begin{proof}
Assume that it is not, hence the intersection of all prime congruences contains a tuple $(x,y)$ where $x\ne y$. Without loss of generality we can assume that $y=1$, and using the above Lemma, we obtain that there is a prime not containing our tuple. This is a contradiction.
\end{proof}

\begin{Theorem}\label{thm:primedec}
Let $C$ be a cancellative congruence in the $\Finfty$-algebra $A$. Then $C$ is the intersection of all prime congruences containing $C$.
\end{Theorem}

\begin{proof}
This follows from Lemma \ref{prop:frac} and Theorem \ref{thm:field}.
\end{proof}

\smallskip

Nikolai Durov in his thesis (\cite{durov2007}) defines a \emph{family of spectra}. The minimal spectrum, the one he refers to as the \emph{unary spectrum} (\textbf{6.1}) is a straight-forward generalization of the spectrum given by prime ideals. Theorem \ref{thm:primedec} shows that our prime spectrum, given by prime congruences, is a refinement of Durov's unary spectrum, as every prime ideal gives rise to a congruence generated by it that is cancellative.


\begin{thebibliography}{XXX}

\bibitem{arakelov1974}
Suren Yurevich Arakelov. \emph{Intersection theory of divisors on an arithmetic surface.} Math. USSR Izv (1974), \textbf{8} (6): 1167--1180.

\bibitem{arakelov1975}
Suren Yurevich Arakelov. \emph{Theory of intersections on an arithmetic surface.} Proc. Internat. Congr. Mathematicians Vancouver, 1974, \textbf{1}, Amer. Math. Soc., pp. 405--408. (1975)

\bibitem{durov2007}
Nikolai Durov. \emph{New Approach to Arakelov Geometry.} PhD Thesis. 2007.

\bibitem{Fal}
Gerd Faltings. \emph{Diophantine Approximation on Abelian Varieties}, Annals of Mathematics (1991), Second Series, 133 (3): 549--576.

\bibitem{joo2014}
D\'aniel Jo\'o, Kalina Mincheva. \emph{Prime congruences of idempotent semirings and a Nullstellensatz for tropical polynomials.} Selecta Mathematica, pp 1--27, (2017).

\bibitem{joo2015}
D\'aniel Jo\'o, Kalina Mincheva. \emph{On the dimension of polynomial semirings.} Preprint. 2015.

\bibitem{Voj} Paul Vojta. \emph{Siegel's Theorem in the Compact Case}, Annals of Mathematics, (1991) Vol. 133, No. 3, 133 (3): 509--548.

\end{thebibliography}
\end{document}